\documentclass[12pt]{amsart}

\usepackage{amsthm, amsmath, amssymb, amsrefs}

\usepackage{fullpage}

\usepackage{hyperref}

\newcommand{\T}{\mathbb{T}} 
\newcommand{\C}{\mathbb{C}} 
\newcommand{\D}{\mathbb{D}} 
\newcommand{\Z}{\mathbb{Z}} 

\newcommand{\mcS}{\mathcal{S}}
\newcommand{\mcA}{\mathcal{A}}
\newcommand{\mcH}{\mathcal{H}}
\newcommand{\mcL}{\mathcal{L}}

\numberwithin{equation}{section}

\newtheorem{theorem}{Theorem}[section]
\newtheorem{prop}[theorem]{Proposition}
\newtheorem{corollary}[theorem]{Corollary}
\newtheorem{lemma}[theorem]{Lemma}

\theoremstyle{definition}
\newtheorem{definition}[theorem]{Definition}

\newtheorem{question}[theorem]{Question}
\newtheorem{remark}[theorem]{Remark}

\title{Kummert's approach to realization on the bidisk}

\author{Greg Knese}
\address{Washington University in St. Louis\\ Department of Mathematics \& Statistics \\
St. Louis, MO 63130}
\email{geknese@wustl.edu}

\thanks{Partially supported by NSF grant DMS-1900816}

\date{\today}

\begin{document}

\begin{abstract}
We give a simplified
exposition of Kummert's approach to 
proving that every matrix-valued rational
inner function in two variables
has a minimal unitary transfer function realization.
A slight modification of the approach
extends to rational functions
which are isometric on the two-torus
and we use this to 
give a largely elementary new proof
of the existence of Agler decompositions
for every matrix-valued Schur function in
two variables. 
We use a recent result of
Dritschel to prove two variable
matrix-valued rational Schur functions
always have finite-dimensional contractive 
transfer function realizations.
Finally, we prove that two variable
matrix-valued polynomial
inner functions have transfer function
realizations built out of special nilpotent linear
combinations.
\end{abstract}

\maketitle

\section{Introduction}

The goal of this paper is to give a simple proof
and several applications of the following theorem.

\begin{theorem}[Main Theorem] \label{mainthm}
Assume $S: \D^2 \to \C^{M\times N}$ is
rational with no poles in $\D^2$ and satisfies $S^*S = I_N$
on $\T^2$ away from the zero set of the denominator of $S$.

Then, there exist an integer $r$ and an $(M+r)\times (N+r)$ isometric matrix 
$U = \begin{pmatrix} A & B \\ C & D \end{pmatrix}$
such that
\begin{equation} \label{STFR}
S(z) = A + B \Delta(z) (I-D \Delta(z))^{-1} C
\end{equation}
where $\Delta(z_1,z_2) = z_1 P_1 + z_2 P_2$ and $P_1,P_2$ are
orthogonal projections with $P_1+P_2 = I_r$.  
\end{theorem}

Above $\D^2 = \{z=(z_1,z_2)\in \C^2: |z_1|,|z_2|<1\}$ is
the unit bidisk and $\T^2 = \{(z_1,z_2)\in \C^2: |z_1|=|z_2|=1\}$
is the two-torus (or bitorus).  
We shall call functions that satisfy the hypotheses
of this theorem \emph{rational iso-inner functions}.
Formulas in the conclusion of this theorem such as \eqref{STFR}, 
which are built out of block operators, 
will be called \emph{transfer function realizations} (or TFRs).
If the operator is a finite matrix we will call it a finite TFR
and if we have extra information about the operator involved
we will incorporate it into the terminology.  For example, the
above theorem asserts the existence of a ``finite isometric TFR''
for two variable rational iso-inner functions.

This theorem is due to Kummert
in the square case $M=N$ \cite{Kummert89}.  
Kummert's theorem was ahead of its time and its
proof was both ingenious and largely elementary.
At the same time, Kummert's argument 
\emph{seems} complicated
and the engineering terminology
may obscure the underlying concepts for some,
so one of our main goals is to give a 
simplified, conceptual, and entirely mathematical account of Kummert's approach.
We also give an algorithm for constructing
the matrix $U$.
Motivation for doing so comes from
recent interest in the wavelet community in transfer
function formulas in one and several variables \cite{CCCP}.
We have presented generalizations
of our simplified argument in a couple of papers 
\cites{K11, GIK16},
but the generalizations can also
potentially obscure the underlying concepts.
A minor adjustment allows us to treat the non-square case
$M\ne N$, which in turn allows us to give
possibly the most elementary and direct
proof of the following seminal theorem of Agler.

\begin{theorem}[Agler \cites{Agler1,Agler2}] \label{Aglerthm}
Let $f:\D^2\to \C^{M\times N}$ be holomorphic
and $\|f(z)\|\leq 1$ for all $z\in \D^2$.
Then, $f$ has a contractive TFR:
there exists a contractive operator $T$ on 
some Hilbert space with block
decomposition
$T = \begin{pmatrix} A& B \\ C & D\end{pmatrix}$
such that
\[
f(z) = A + B \Delta(z) (I-D\Delta(z))^{-1} C
\]
where $\Delta(z) = z_1 P_1+z_2 P_2$ and $P_1,P_2$ 
are pairwise orthogonal orthogonal projections which
sum to the identity on the domain of $D$.
\end{theorem}

Perhaps, the most important application
of this theorem is a Pick interpolation
theorem for holomorphic functions
on the bidisk.
For this and other applications 
we refer the reader to 
the book \cite{AMbook} and the papers
\cites{AMcrelle, AMYmonotone, AMYcara, BT98}.

Dritschel has recently proven a strong Fej\'er-Riesz
type of result in two variables (Theorem \ref{dritschel2}) 
which makes it possible to 
prove that every two-variable rational function 
bounded by one in norm on $\D^2$
(with no assumptions 
on boundary behavior) 
has a \emph{finite contractive} 
TFR.  

\begin{theorem}\label{finitetfr}
Let $S: \D^2\to \C^{M\times N}$ be
rational with no poles in $\D^2$ and assume 
$\|S(z)\|\leq 1$ for all $z\in \D^2$.
Then, there exists a contractive matrix
$T = \begin{pmatrix} A & B \\ C & D \end{pmatrix}$
such that
\[
S(z) = A + B \Delta(z) (I-D \Delta(z))^{-1} C
\]
where $\Delta(z_1,z_2) = z_1 P_1 + z_2 P_2$, $P_1,P_2$ are
orthogonal projections with $P_1+P_2 = I$.  
\end{theorem}

A very important bonus of 
Kummert's approach is that it constructs the matrix
$U$ in Theorem \ref{mainthm}
 with the minimal possible dimensions in
a strong way.  
For a rational iso-inner function $S:\D^2 \to \C^{M\times N}$ 
we can always make sense of $z_1 \mapsto S(z_1,z_2)$
for each fixed $z_2\in \T$ and this is a one variable
rational iso-inner function (Lemma \ref{Hurwitzlemma}).  
If we have a formula as in Theorem \ref{mainthm} where
the ranks of $P_1,P_2$ are $r_1,r_2$ then we can construct
a transfer function realization for $S(\cdot,z_2)$ with size $r_1$ and a transfer
function realization for
$S(z_1,\cdot)$ with size $r_2$.  In the square case $M=N$, this
can be done optimally.

\begin{theorem}[Kummert's minimality theorem] \label{minthm} 
Suppose $S:\D^2\to \C^{N\times N}$ is rational and inner.
Then, one can choose $U$ in Theorem \ref{mainthm} so 
that the ranks $r_1,r_2$ of $P_1, P_2$ are
simultaneously minimal:
$r_1$ is the maximum of the minimal size of 
a unitary TFR for $z_1 \mapsto S(z_1,z_2)$ where $z_2$ varies over $\T$
and $r_2$ is the maximum of the minimal size of a
unitary TFR for $z_2 \mapsto S(z_1,z_2)$ where $z_1$ varies over $\T$.
\end{theorem}

In particular, among all possible unitary TFR's for $S$, neither $r_1$ nor $r_2$ can
be smaller than those in Kummert's construction.
We will give a conceptual
proof of Kummert's minimality theorem, and
clarify why this is the best
possible result.
Before the mathematical community knew of
Kummert's results, this result was reproven in
the scalar case using the framework of Geronimo-Woerdeman \cite{GW04}
in \cite{GKAPDE}.  
Later, Theorem \ref{minthm} was also proven using Hilbert space methods
in \cite{BickelKnese}.
The scalar minimality theorem was crucial in giving a characterization
of two-variable rational matrix-monotone functions in \cite{AMYmonotone}.
It is also useful in proving determinantal representations for
certain families of polynomials $p\in \C[z_1,z_2]$ with no zeros in $\D^2$ \cite{GKdv}.

We shall present a new application of the minimality
theorem which has some relevance to the applications
of this theory to wavelets in \cites{wavelet, CCCP}.
In these papers matrix-valued \emph{polynomial} inner
functions are of particular interest.  
\begin{theorem} \label{thmpolyinner}
Let $S \in \C^{N\times N}[z_1,z_2]$ 
and assume $S^*S = I_N$ on $\T^2$.
Then,
$U$ in Theorem \ref{mainthm} can be chosen
with $\det(I- D\Delta(z)) \equiv 1$.  
\end{theorem}

Note this means $D\Delta(z) = z_1DP_1+z_2DP_2$ is nilpotent
for every $z$.

\subsection{Guide to the reader}
This paper is structured so that it can hopefully
be read by a broad audience.
We make no mention of systems theory terminology
(except for ``transfer function'') and we make no 
use of von Neumann inequalities and related operator
theory originally used in the proof of Agler's theorem.
(We do discuss some of this for context in Section \ref{sec:AD}.)  
Our first goal is to quickly and simply prove Kummert's 
Theorem \ref{mainthm} and explain how
this proves Agler's theorem.  
Some readers may be satisfied with this quick and
mostly constructive approach to these results
and can stop after Section \ref{sec:AD}.
After that we introduce the technicalities necessary
to prove Kummert's minimality theorem and 
give an application to inner polynomials.
We include an appendix with extra background.

\subsection{Acknowledgments}
This article overlaps with the interesting article of J. Ball \cite{Ball} 
in some ways:  both survey Agler decompositions
on the bidisk/polydisk but Ball's article follows Kummert's original argument closely.
Ball's paper also discusses connections to the
engineering literature and several other classes
holomorphic functions.
The present article and author owe a great debt to Professor Ball 
for disseminating Kummert's argument to the mathematical
community.  

This article was motivated by 
the workshop 
``Mathematical Challenges of Structured
Function Systems'' 
at the Erwin Schr\"odinger Institute.
I thank ESI as well as the workshop
organizers (M. Charina, K. Gr\"ochenig,
M. Putinar, and J. St\"ockler).
The article \cite{wavelet} was helpful 
in preparing this paper.
I thank M. Dritschel for reading an early
draft of this paper.
I also thank K. Bickel
for suggesting to me to write this paper.
Finally, I sincerely thank the referee for several suggestions which
greatly improved this paper.

\tableofcontents

\section{Finite-dimensional transfer function realizations}\label{sec:equiv}

One of the fundamental things that Agler did in his original
proof of Theorem \ref{Aglerthm} 
was connect TFRs to certain formulas now called \emph{Agler decompositions}
which involved positive
semi-definite kernels.
The following theorem establishes
some basic equivalences about \emph{finite} TFRs and
\emph{finite}-dimensional Agler decompositions
which hold not just on $\D^2$ but
any polydisk $\D^d$.  Note that  ``matrix'' below always refers
to a finite matrix.  

\begin{theorem}[Equivalences Theorem] \label{equivalences}
Let $S:\D^d\to \C^{M\times N}$ be a function.

The following are equivalent:
\begin{enumerate}
\item There exists a contractive matrix $T = \begin{pmatrix} A & B \\ C & D \end{pmatrix}$ such that
\[
S(z) = A + B\Delta(z) (I-D \Delta(z))^{-1} C
\]
where $\Delta(z) = \sum_j z_j P_j$, 
for some pairwise orthogonal projections with 
$\sum_j P_j = I$.
\item There exist matrix functions $F_j$ and a constant contractive matrix $T$ such that
\[
T \begin{pmatrix} I \\ z_1 F_1(z) \\ \vdots \\ z_d F_d(z) \end{pmatrix} = \begin{pmatrix} S(z) \\ F_1(z) \\ \vdots \\ F_d(z) \end{pmatrix}.
\]
\item
There exist matrix functions $F_1,\dots, F_d, G$ such that
\[
I-S(w)^* S(z) = G(w)^*G(z) + \sum_j (1-\bar{w}_j z_j) F_j(w)^* F_j(z).
\]
\end{enumerate}

We also have the following bonuses:
\begin{description}
\item[B1] Assuming (1)-(3), $S$, $F_1,\dots, F_d, G$ are all rational 
and $\|S(z)\| \leq 1$ for all $z\in \D^d$.
If we assume at the outset that $S$ is holomorphic,
then item (3) need only hold initially on an open set
in order for it to hold globally.
\item[B2] The $T$ that works in (1) also works in (2).

\item[B3] We also get equivalences if we replace ``contractive'' 
in (1) and (2) with ``isometric'' and $G$ with $0$ in (3).
In this case, $S$ is iso-inner and analytic outside the zeros of $\det(I-D\Delta(z))$.

\end{description}

\end{theorem}

\begin{proof}
$(2)\implies (1)$. 
It helps to define $F(z) = \begin{pmatrix} F_1(z) \\ \vdots \\ F_d(z) \end{pmatrix}$.
Let $P_j$ be the projection matrix for the block corresponding to $F_j$. 
Then, the equation in (2) can be written as
\begin{equation} \label{21}
\begin{pmatrix}  A & B \\ C & D \end{pmatrix} \begin{pmatrix} I & 0 \\ 0 & \Delta(z) \end{pmatrix} 
\begin{pmatrix} I \\ F(z) \end{pmatrix} = \begin{pmatrix} S(z) \\ F(z) \end{pmatrix}
\end{equation}
for $\Delta(z) = \sum_j z_j P_j$.
Block-by-block this says
\[
\begin{aligned}
A + B \Delta F &= S \\
C + D \Delta F &= F
\end{aligned}
\]
which yields $F = (I-D\Delta)^{-1} C$ and then $S = A+ B \Delta (I-D\Delta)^{-1} C$.

$(1)\implies (2)$.  We simply define $F = (I-D \Delta)^{-1} C$.
Then, \eqref{21} holds because
\[
C + D\Delta (I-D \Delta)^{-1} C =  (I-D\Delta)^{-1}C.
\]

$(2)\implies (3)$.  The given equation implies
\[
\begin{pmatrix} I \\ \Delta(w) F(w) \end{pmatrix}^* T^*T \begin{pmatrix} I \\ \Delta(z) F(z) \end{pmatrix} = 
\begin{pmatrix} S(w) \\ F(w) \end{pmatrix}^* \begin{pmatrix} S(z) \\ F(z) \end{pmatrix}.
\]
Let $A = \sqrt{I-T^*T}$ and $G(z) = A \begin{pmatrix} I \\ \Delta(z) F(z) \end{pmatrix}$.
Then,
\[
\begin{pmatrix} I \\ \Delta(w) F(w) \end{pmatrix}^* \begin{pmatrix} I \\ \Delta(z) F(z) \end{pmatrix} = 
\begin{pmatrix} S(w) \\ F(w) \end{pmatrix}^* \begin{pmatrix} S(z) \\ F(z) \end{pmatrix} + G(w)^*G(z)
\]
and this rearranges exactly into the equation in (3).

$(3) \implies (2)$.  This is known as a lurking isometry argument.
The map
\[
\begin{pmatrix} I \\ \Delta(z) F(z) \end{pmatrix} \mapsto \begin{pmatrix} S(z) \\ F(z) \\ G(z) \end{pmatrix}
\]
extends linearly and in a well-defined way to 
an isometric map from the span of the vectors on the left to 
the span of the vectors on the right as $z$ varies over $\D^d$.
We can extend this to an isometric matrix $V$ satisfying
\[
V \begin{pmatrix} I \\ \Delta(z) F(z) \end{pmatrix} = \begin{pmatrix} S(z) \\ F(z) \\ G(z) \end{pmatrix}
\]
which we can compress to get a 
contractive matrix satisfying the equation in (2).

The bonus results follow. For (B1), $S$ is rational and
bounded in operator norm by $1$ by (1) and (3).  
The matrix functions $F_j,G$ are rational by 
the proofs of $(2)\implies (1)$ and $(2) \implies (3)$.
If we assume $S$ is holomorphic and (3) only
holds on an open set, then all of the
proofs work on this restricted set
but automatically extend holomorphically
to $\D^d$ by the matrix formulas.  
Bonus (B2) follows from the proof of $(1) \iff (2)$.
For bonus (B3), notice that if $T$ is an isometric matrix, 
then we have $G=0$ in the proof $(2)\implies (3)$
and if we start with $G=0$ we get $T$ to be isometric in 
the proof $(3)\implies (2)$ since
no compression is necessary.
Finally, $S$ is iso-inner because we can insert $z=w\in \T^d$ into
condition (3) to see $S^*S = I$ at least
away from the zero set of $\det(I-D \Delta(z))$ which is a denominator
for the $F_j$ and $S$ by the formula in (2)$\implies$ (1).
\end{proof}

The next proposition says 
the conditions of Theorem \ref{equivalences} are also
equivalent to $S$ being a submatrix 
of a rational inner function
possessing a finite-dimensional unitary transfer function realization.
Moreover, the various sizes of the transfer function
realizations stay the same.
To be more precise, 
let $r_j$ be the rank of $P_j$ in 
condition (1) of Theorem \ref{equivalences}.
Then, $r=(r_1,\dots,r_d)$ will be called 
the \emph{size breakdown} of the TFR.
This terminology is endemic to 
this paper.
The \emph{size} of the TFR will refer to 
$|r| = r_1+\cdots + r_d$.
Note that $r_j$ also equals the number of rows
of $F_j$ in conditions (2) and (3) 
of Theorem \ref{equivalences}.

\begin{prop}\label{partthm}
Let $S:\D^d\to \C^{M\times N}$ be a function 
which has a finite contractive
TFR with size breakdown $r$.
Then, there exists $n\geq N,M$ and a matrix rational inner function $\Phi:\D^d \to \C^{n\times n}$
with finite unitary TFR
with size breakdown $r$
such that $S$ is a submatrix of $\Phi$. 

As a sort of converse, every submatrix of $S$ has a finite contractive TFR with
same size breakdown.
\end{prop}

\begin{proof}
Suppose $S$ has a finite contractive TFR given via
contractive $T = \begin{pmatrix} A & B \\ C & D \end{pmatrix}$.
Every contractive matrix is a submatrix of a finite unitary, say $U$.
If we rearrange rows and columns we may write 
\[
U = \begin{pmatrix} A & A_{12} & B \\
A_{21} & A_{22} & B_2 \\
C & C_2 & D \end{pmatrix}.
\]
If 
\[
\Phi(z) = \begin{pmatrix} A & A_{12} \\ A_{21} & A_{22} \end{pmatrix} +
 \begin{pmatrix} B \\ B_2\end{pmatrix}\Delta(z) (I-D \Delta(z))^{-1} \begin{pmatrix} C & C_2 \end{pmatrix}
 \]
then $S(z) = \begin{pmatrix} I & O \end{pmatrix} \Phi(z) \begin{pmatrix} I \\ O \end{pmatrix}$.

This same type of observation shows that every submatrix of $S$ has a 
finite contractive TFR.
\end{proof}

The following is referred to as the adjunction formula
in \cite{wavelet}.

\begin{prop}\label{breve}
Let $S:\D^d \to \C^{M\times N}$ be a function
with a finite contractive TFR given via a matrix $T$ 
as in (1),(2) of Theorem \ref{equivalences}.
Set $\breve{S}(z) = S(\bar{z})^*$. 
Then, $\breve{S}$ has a finite contractive
TFR given via $T^*$.

In particular, if $T$ is isometric, then $\breve{S}$
has a finite coisometric TFR.  
\end{prop}

\begin{proof}
With $S(z) = A+B\Delta(z)(I-D\Delta(z))^{-1}C$
we have
\[
\begin{aligned}
\breve{S}(z) &= A^* + C^*(I-\Delta(z)D^*)^{-1} \Delta(z) B^* \\
&= A^* + C^*\Delta(z) (I- D^*\Delta(z))^{-1} B^*
\end{aligned}
\]
which is exactly condition (1) of Theorem \ref{equivalences}
with $T^*$ in place of $T$.
\end{proof}

\section{One variable version of Theorem \ref{mainthm}}\label{sec:onevar}
 
We now prove a detailed one variable version of the Main Theorem (Thm \ref{mainthm}).
If $S = Q/p:\D\to \C^{M \times N}$ is a rational iso-inner function,
then
 $S^*S = I$ on $\T$ away from zeros 
of $p$, but then $|p|^2I = Q^*Q$ on all of $\T$ 
by continuity.  

\begin{theorem}\label{onevarthmdetailed}
Assume $p \in \C[z]$ has
no zeros in $\D$, 
$Q\in \C^{M\times N}[z]$,  and $|p|^2 I = Q^*Q$ on $\T$.
Let $n$ be the maximum of the degrees
of $p$ and the entries of $Q$.
Then,
\begin{equation} \label{pkernel}
K(w,z) = \frac{\overline{p(w)}p(z)I - Q(w)^* Q(z)}{1-\bar{w}z} = (I, \bar{w}I,\dots, \bar{w}^{n-1} I) T (I, zI, \dots, z^{n-1} I)^t
\end{equation}
where $T$ is a positive semi-definite matrix whose entries
can be expressed as polynomials in the coefficients of $p, \bar{p}, Q, Q^*$.
Furthermore, $K(w,z)$ is a positive semi-definite kernel
whose rank 
matches the rank of the matrix $T$.
\end{theorem}

Positive semi-definite kernels are reviewed
in Definition \ref{def:PSD} and the rank of such a kernel
is defined in Definition \ref{rank} in the Appendix.

The theorem allows for common zeros of $Q$ and $p$ 
which is important in using this result in two variables.
It immediately follows that $S=Q/p$ possesses an isometric TFR
because we can factor $T= \mathbf{F}^* \mathbf{F}$ where $\mathbf{F}$
is an $r\times nN$ matrix.  
Then, for $F(z) = \mathbf{F} (I,zI,\dots, z^{n-1}I)^t$
we have
\[
I-S(w)^*S(z) = (1-\bar{w}z) \left(\frac{F(w)}{p(w)}\right)^*  \frac{F(z)}{p(z)}.
\]
By Theorem \ref{equivalences} we see that $S$ has 
an isometric TFR.  After the proof of Theorem \ref{onevarthmdetailed}
we give an explicit way to find a formula for
an isometry $U$ out of which a TFR for $S$ can be built.
We need a standard lemma to prove Theorem \ref{onevarthmdetailed}.  
We give the short proof in the appendix; see Subsection \ref{PSDsec}.

\begin{lemma}\label{Skernel-lemma}
Assume $S:\D\to \C^{M\times N}$ is analytic and $\|S(z)\| \leq 1$ in $\D$.  Then,
the kernel
\begin{equation} \label{Skernel}
K_S(w,z) = \frac{I-S(w)^*S(z)}{1-\bar{w}z}
\end{equation}
is positive semi-definite.  
\end{lemma}

The swapping of 
$z$,$w$ is deliberate and is discussed in the proof in the
appendix.
\begin{proof}[Proof of Theorem \ref{onevarthmdetailed}]
By analyticity $\overline{p(1/\bar{z})}p(z) I = Q(1/\bar{z})^* Q(z)$
on $\C\setminus \{0\}$.
This implies the polynomial in $z,\bar{w}$
\[
\overline{p(w)}p(z) I - Q(w)^* Q(z)
\]
is divisible by $(1-\bar{w}z)$ and hence we can write
\eqref{pkernel} where $T$ is indeed a $nN\times nN$ matrix
whose entries are polynomials 
 in the coefficients
of $p,\bar{p}, Q, Q^*$.  We could solve for them
but we do not need to.
By Lemma \ref{Skernel-lemma}, $K_S(w,z)$ in \eqref{Skernel} is positive
semi-definite.  Multiplying through by
$\overline{p(w)}p(z)$ we have that $K(w,z)$ as in \eqref{pkernel}
is a positive semi-definite
matrix-valued polynomial function of
bounded degree.

To show $T$ is positive semi-definite, take
any $z_1,\dots, z_n \in \D$ and note that
\[
(K(z_i,z_j))_{i,j} = \begin{pmatrix} I & \bar{z}_1 I& \cdots & \bar{z}_1^{n-1} I \\
I & \bar{z}_2 I & \cdots & \bar{z}_2^{n-1} \\
\vdots & \vdots & \ddots & \vdots \\
I & \bar{z}_n I & \cdots & \bar{z}_n^{n-1} I \end{pmatrix}
T 
\begin{pmatrix} I & I & \cdots & I \\
z_1I & z_2 I & \vdots & z_n I \\
\vdots & \vdots & \ddots & \vdots \\
z_1^{n-1} I & z_2^{n-1} & \cdots & z_n^{n-1} I \end{pmatrix}
=
V^* T V
\]
is positive semi-definite
where $V = (V_{i,j})$ is the block Vandermonde matrix $V_{i,j} = z_j^{i-1} I$.
If the $z_j$ are all distinct then $V$ is invertible which implies that $T$
is positive semi-definite.
The above computation also shows that the rank of $K$ equals the rank
of $T$, although we omit some details.

\end{proof}

\begin{remark}\label{TFRformula1}
We now explain how to find an isometry 
$U$ out of which
a TFR for $S = Q/p$ can be built.
This will closely parallel our
approach in the two variable setting. 
We first factor 
$T = A^*A$ where $A$ is $r\times nN$
with $r=\text{rank}(T)$.  
Then, $A$ will possess a right inverse $B$, 
namely $AB=I$.  
Set $F(z) = A(I,zI,\dots, z^{n-1}I)^t$.  
To find $U$ such that
\[
U \begin{pmatrix} p(z)I \\ zF(z) \end{pmatrix} = \begin{pmatrix} Q(z) \\ F(z)
\end{pmatrix}
\]
we write out $p(z) = \sum_{j=0}^{n} p_jz^j, Q(z) = \sum_{j=0}^{n} z_j Q_j$
and extracting coefficients we equivalently need $U$ to satisfy
\[
U \begin{pmatrix} p_0 I & \left[p_1 I , \dots , p_n I\right] \\
O & A\end{pmatrix}
=
\begin{pmatrix} \left[Q_0 , \dots , Q_{n-1}\right] & Q_n\\
A & O \end{pmatrix}.
\]
The matrix $\begin{pmatrix} p_0 I & \left[p_1 I , \dots , p_n I\right]\\
O & A\end{pmatrix}$
has right inverse
\[
\begin{pmatrix} p_0^{-1} I & X \\
O & B \end{pmatrix}
\]
where $X = -p_0^{-1}\left[p_1I,\dots, p_nI\right] B$
so that
\[
U = \begin{pmatrix} \left[Q_0 , \dots , Q_{n-1}\right] & Q_n\\
A & O \end{pmatrix} 
\begin{pmatrix} p_0^{-1} I & X \\
O & B \end{pmatrix}.
\]
Thus, $U$ can
be computed directly
from $p,Q, A, B$.
\end{remark}

\section{Two variables and Theorem \ref{mainthm}}\label{sec:twovar}
The basic idea of Kummert's argument
is to attempt a parametrized version of the
one variable theorem above.  
The matrix Fej\'er-Riesz factorization in one variable, which
we now review,
then becomes crucial in attempting a parametrized 
version of the implication (3)$\implies$ (2)
in the Equivalences Theorem (Thm \ref{equivalences}).

\begin{theorem}[Matrix Fej\'er-Riesz] \label{genFR}
Let $T(z) = \sum_{j=-n}^{n} T_j z^j$ be a matrix Laurent polynomial ($T_j \in \C^{N\times N}$)
such that $T(z) \geq 0$ for $z \in \T$.
Then, there exist a natural number $r\leq N$, a matrix polynomial $A_0 \in \C^{r\times r}[z]$
with $\det A_0(z) \ne 0$ for $z\in \D$, and a polynomial matrix $V \in \C^{N\times N}[z]$ 
with polynomial inverse 
such that for $A = \begin{pmatrix} A_0 & 0_{r\times N-r} \end{pmatrix} V$ we have
\[
T = A^*A \text{ on } \T.
\]
Furthermore, $A$ has degree at most $n$ and a right rational inverse $B$ which is analytic in $\D$.
\end{theorem}

The case where $T(z)$ is positive definite at all points
of $\T$ is usually attributed to Rosenblatt \cite{Rosenblatt}.
If $\det T(z)$ vanishes at a finite number
points, it is possible to factor out these zeros from $T$; see \cites{DGK,Dj}.
If $\det T(z)$ is identically zero, it is possible to use operator-valued 
versions of this theorem which guarantee an \emph{outer} 
factorization of $T$.  
We explain how to go from the case of $\det T \not\equiv 0$ to
the case $\det T \equiv 0$ in the appendix (subsection \ref{Fejerproofs}).
The factorization above can be computed
using semidefinite programming or 
Riccati equations (see for instance \cite{Hachez}).

Theorem \ref{genFR} in particular shows that 
$T(z)$ has rank $r$ except at 
the finite number of zeros of $\det A_0$.  
One nice application of Theorem \ref{genFR} is 
the one variable version of Theorem \ref{finitetfr}.

\begin{prop} \label{onevarfinitetfr}
 Let $S: \D\to \C^{M\times N}$ be rational and $\|S(z)\|\leq 1$ for all $z\in \D$.
Then, $S$ has a finite contractive TFR.
\end{prop}
\begin{proof}
Write $S = Q/p$.  Then, $|p|^2 I -Q^*Q$ is positive
semi-definite on $\T$.  By Theorem \ref{genFR}, there
exists a matrix polynomial $A$ such that $|p|^2-Q^*Q = A^*A$
on $\T$.  Then, $\Phi = \begin{pmatrix} S \\ A/p \end{pmatrix}$
is iso-inner and by Theorem \ref{onevarthmdetailed} possesses
a finite isometric TFR.  By Proposition \ref{partthm}, we see that 
$S$ possesses a finite contractive TFR.
\end{proof}

The
following lemma lets us apply Theorem \ref{onevarthmdetailed}
to one variable slices.

\begin{lemma} \label{Hurwitzlemma}
Suppose $S: \D^2 \to \C^{M\times N}$ is rational and iso-inner.
Write $S = Q/p$ where $Q \in \C^{M\times N}[z_1,z_2]$, $p\in \C[z_1,z_2]$
has no zeros in $\D^2$, and $Q$, $p$ have no common factors.
Then, $|p|^2I = Q^*Q$ on $\T^2$ and for each $z_2 \in \T$, 
the one variable polynomial $z_1 \mapsto p(z_1,z_2)$ 
has no zeros in $\D$.
\end{lemma}

\begin{proof}
As in one variable, $|p|^2 I = Q^*Q$ on $\T^2$ by continuity.
For fixed $\tau \in \T$ notice
that $z_1\mapsto p(z_1,\tau)$ either has no 
zeros in $\D$ or is identically zero by Hurwitz's theorem
(by considering $\tau$ as a limit of $t\in \D$).
If $p(\cdot,\tau)$ is identically zero, then $Q(\cdot,\tau)$ is identically 
zero because of $|p|^2I = Q^*Q$ on $\T^2$.  
Hence both polynomials are divisible by $z_2-\tau$
contradicting the assumption of no common factors.
Thus, for every $z_2\in \T$, $z_1\mapsto p(z_1,z_2)$ has
no zeros in $\D$.
\end{proof}

We are now ready to prove the Main Theorem (Thm \ref{mainthm}).

\begin{proof}[Proof of Theorem \ref{mainthm}]
Assume the setup of Theorem \ref{mainthm} and write $S=Q/p$ 
as in Lemma \ref{Hurwitzlemma}.
We can essentially follow a parametrized
version of Remark \ref{TFRformula1}
but we use the matrix Fej\'er-Riesz theorem
to deal with certain matrix factorizations.

\underline{Step 1}: 
Fix $z_2=w_2 \in \T$, divide $\overline{p(w)}p(z)I - Q(w)^*Q(z)$ by
$(1-\bar{w}_1z_1)$, and then
extract the coefficients of $\bar{w}_1^jz_1^k$
to obtain
\begin{equation} \label{paraSC}
\frac{\overline{p(w)}p(z)I - Q(w)^* Q(z)}{1-\bar{w_1} z_1} = 
\sum_{j,k} \bar{w}_1^j z_1^k T_{jk}(z_2) = 
(I, \bar{w}_1I,\dots, \bar{w}_1^{n_1-1}I) T(z_2) \begin{pmatrix} I \\ z_1 I \\ \vdots \\ z_1^{n_1-1} I \end{pmatrix}
\end{equation}
where $T(z_2) = (T_{jk}(z_2))_{jk}$ is a positive semi-definite $(n_1N\times n_1N)$ matrix Laurent polynomial. This follows
from Theorem \ref{onevarthmdetailed} applied to $p(\cdot,z_2), Q(\cdot, z_2)$.
Here $n_1$ is the maximum of the degree of $p,Q$ with respect to $z_1$.

\underline{Step 2}: 
Apply the matrix Fej\'er-Riesz theorem (Thm \ref{genFR}) to $T(z_2)$
to get
an $r\times n_1N$ matrix polynomial
$A(z_2)$ and an analytic (in $\D$) rational matrix
function $B(z_2)$ such that $A^*A = T$ on $\T$ 
and $AB = I$ in $\D$.
For convenience we define
\[
\Lambda(z_1) = (I_{N}, z_1 I_{N},\dots, z_1^{n_1-1}I_{N})^t \in \C^{n_1N\times N}[z_1].
\]
Then, for $z_2=w_2\in \T$ and $z_1,w_1 \in \C$
\[
\overline{p(w)}p(z)I_N - Q(w)^* Q(z) 
= (1-\bar{w}_1 z_1) \Lambda(w_1)^* A(w_2)^*A(z_2) \Lambda(z_1).
\]
By Lemma \ref{Hurwitzlemma}, 
for each fixed $z_2\in \T$ the map
$z_1\mapsto \frac{Q(z_1,z_2)}{p(z_1,z_2)}$
is an iso-inner rational function and Theorem \ref{equivalences} guarantees
the existence of an isometric matrix $U(z_2)$
such that
\begin{equation} \label{Ueq0}
U(z_2) \begin{pmatrix} p(z)I_N \\ z_1 A(z_2) \Lambda(z_1) \end{pmatrix} = \begin{pmatrix} Q(z) \\ A(z_2) \Lambda(z_1) \end{pmatrix}.
\end{equation}

\underline{Step 3}: 
In this step we find a formula for $U(z_2)$
and show it extends to $\overline{\D}$ as a rational iso-inner
function in one variable.
We can rewrite \eqref{Ueq0} in terms of the coefficients of the powers of $z_1$ by
writing $p(z) = \sum_{j} p_j(z_2) z_1^j$ and $Q(z) = \sum_j Q_j(z_2) z_1^j$,
defining $\vec{p}(z_2) = (p_0(z_2)I_N, p_1(z_2)I_N, \dots, p_{n_1}(z_2)I_N)$,
and $\vec{Q}(z_2) = (Q_0(z_2), \dots Q_{n_1}(z_2))$.
Then,
\begin{equation} \label{Ueq}
U(z_2) \begin{pmatrix} \vec{p}(z_2) \\ \begin{matrix} O_{r\times N} & A(z_2) \end{matrix} \end{pmatrix} 
= \begin{pmatrix} \vec{Q}(z_2) \\ \begin{matrix} A(z_2) & O_{r\times N} \end{matrix} \end{pmatrix}
\end{equation}
using $O_{r\times N}$ to denote the $r\times N$ zero matrix.
Since $p(0,z_2)= p_0(z_2)$ has no zeros in $\D$, the matrix 
$\begin{pmatrix} \vec{p}(z_2) \\ \begin{matrix} 0 & A(z_2) \end{matrix} \end{pmatrix}$
has a rational matrix right inverse of the form $\begin{pmatrix} p_0(z_2)^{-1}I & X(z_2) \\ 0 & B(z_2) \end{pmatrix}$.
The exact formula for $X(z_2)$ is $-\frac{1}{p_0} (p_1 I, \dots, p_{n_1} I) B$.
Then,
\begin{equation} \label{Uformula}
U(z_2) = \begin{pmatrix} \vec{Q}(z_2) \\ \begin{matrix} A(z_2) & 0 \end{matrix} \end{pmatrix}  \begin{pmatrix} p_0(z_2)^{-1}I & X(z_2) \\ 0 & B(z_2) \end{pmatrix}
\end{equation}
extends to a rational function holomorphic in $\D$ and isometry-valued on $\T$ away from any singularities.  
So, not only is $U$ uniquely determined (by $A,B$) and iso-inner but
both sides of \eqref{Ueq} are now holomorphic, so \eqref{Ueq} extends to $\D$.
(We caution that the blocks in \eqref{Uformula} do not line up as written.  
There is no need to multiply this out, so there is no real concern.)

\underline{Step 4}: 
In this step we find an isometric
matrix $V$ such that $S$ has a TFR built out of
$V$.  
It turns out $U(z_2)$ as a one variable function
has a TFR built out of the same isometry $V$.
Indeed, by Theorem \ref{onevarthmdetailed} and Theorem \ref{equivalences}
there exist a constant isometric matrix $V$ and matrix function $F(z_2)$
such that
\[
V \begin{pmatrix} I \\ z_2 F(z_2) \end{pmatrix} = \begin{pmatrix} U(z_2) \\ F(z_2) \end{pmatrix}.
\]
A formula for $V$ can be found via Remark \ref{TFRformula1}.
As we now show, $V$ is the isometry we are looking for.
If we multiply on the right by $\begin{pmatrix} p(z)I \\ z_1 A(z_2) \Lambda(z_1) \end{pmatrix}$
and define $H(z) := F(z_2) \begin{pmatrix} p(z)I \\ z_1 A(z_2) \Lambda(z_1) \end{pmatrix}$, $G(z) := A(z_2) \Lambda(z_1)$
we get
\[
V \begin{pmatrix} p(z)I \\ z_1 G(z) \\ z_2 H(z) \end{pmatrix} = \begin{pmatrix} Q(z) \\ G(z) \\ H(z) \end{pmatrix}.
\]
By Theorem \ref{equivalences}, this means $S$ has 
a finite-dimensional isometric transfer function realization
built out of the isometry $V$.
This proves Theorem \ref{mainthm}.
\end{proof}

When we prove the minimality theorem (Thm \ref{minthm}) we will pick up where this proof leaves off.
We will later refer to $G^*G$ as the \emph{dominant $z_1$-term}
associated to $S$, while we will refer to $H^*H$ as the \emph{sub-dominant $z_2$-term}.
We write $G^*G:= G(w)^*G(z), H^*H:= H(w)^*H(z)$ instead of $G,H$ because the former are 
uniquely determined while $G,H$ are determined up to left multiplication by isometric matrices.
By symmetry we could also construct a dominant $z_2$-term with associated
sub-dominant $z_1$-term.  

\section{Detailed example}

In this section we give a detailed example
of the 4 steps presented in the proof of Theorem \ref{mainthm}.
The $N\times N$ identity matrix is written $I_N$,
the $N\times N$ zero matrix is written $O_N$, and the $N\times M$
zero matrix is written $O_{N\times M}$.

Consider the following simple rational inner function
\[
S(z) = \frac{1}{2} \begin{pmatrix} z_1(z_1+z_2) & z_1z_2(z_1-z_2) \\
z_1-z_2 & z_2(z_1+z_2) \end{pmatrix} =
 \begin{pmatrix} z_1 & 0 \\ 0 & 1 \end{pmatrix} X \begin{pmatrix} z_1 & 0 \\ 0 & z_2 \end{pmatrix} X
 \begin{pmatrix} 1 & 0 \\ 0 & z_2 \end{pmatrix}
\]
where $X = \frac{1}{\sqrt{2}} \begin{pmatrix} 1 & 1 \\ 1 & -1\end{pmatrix}$ is a unitary.
The right expression shows $S$ is a product of inner functions and
is therefore inner itself.  Since $S$ is a polynomial
the process below will be simpler than the general case
but still illustrative. Note then that referring
to the proof of Theorem \ref{mainthm} we have $p=1$ and $Q=S$.

\underline{Step 1}: 
Set $|z_2|=1$, divide $I-S(w_1,z_2)^*S(z_1,z_2)$
by $1-\bar{w}_1z_1$, and extract coefficients
of the monomials $\bar{w}_1^jz_1^k$ in order to
write
\[
\frac{I-S(w_1,z_2)^*S(z_1,z_2)}{1-\bar{w}_1z_1} = 
\sum_{j,k=0,1} \bar{w}_1^j z_1^k T_{jk}(z_2) 
=
(I_2, \bar{w}_2 I_2) T(z_2) \begin{pmatrix} I_2 \\ z_1 I_2\end{pmatrix}
\]
where $T(z_2)$ is the matrix Laurent polynomial
\[
T(z_2) = \frac{1}{4} \begin{pmatrix} 3 & z_2 & z_2^{-1} & 1 \\
z_2^{-1} & 3 & -z_2^{-2} & -z_2^{-1} \\
z_2 & -z_2^2 & 1 & z_2 \\
1 & z_2 & z_2^{-1} & 1
\end{pmatrix}.
\]
Necessarily, $T$ is positive semi-definite on $\T$.  

\underline{Step 2}: Factor $T$ according to 
the one variable matrix Fej\'er-Riesz theorem.
There exist algorithms for doing this (\cite{Hachez})
and it can also be essentially reduced to 
polynomial algebra and one variable Fej\'er-Riesz
factorizations (see \cite{Dj} where
this is done in a more general setup).
We get $T(z_2) = A(z_2)^* A(z_2)$ on $\T$ where
\[
A(z_2) = \frac{1}{2} \begin{pmatrix} \sqrt{2} & (\sqrt{2})z_2 & 0 & 0 \\
z_2 & -z_2^2 & 1 & z_2 \end{pmatrix} =: (A_0(z_2), A_1(z_2))
\]
has right inverse
\[
B(z_2) = \begin{pmatrix} \sqrt{2} & 0 \\
0 & 0 \\
(-\sqrt{2})z_2 & 2 \\
0 & 0 \end{pmatrix} =: \begin{pmatrix} B_0(z_2) \\ B_1(z_2) \end{pmatrix}.
\]
We use the equations above to define the $2\times 2$ matrix polynomials
$A_0(z_2), A_1(z_2), B_0(z_2) B_1(z_2)$.
Note that the right inverse in general could be rational.

\underline{Step 3}:
We find our parametrized unitary $U(z_2)$
in this step. 
Form
the ``vectors'' of coefficients
\[
\vec{p}(z_2) = (I_2, O_2, O_2) \text{ and } \vec{Q}(z_2) = (Q_0(z_2), Q_1(z_2), Q_2(z_2))
\]
where
\[
Q_0(z_2) = \frac{1}{2}\begin{pmatrix} 0 & 0 \\ -z_2 & z_2^2 \end{pmatrix} ,
Q_1(z_2) = \frac{1}{2}\begin{pmatrix} z_2 & -z_2^2 \\ 1 & z_2 \end{pmatrix},
Q_2(z_2) = \frac{1}{2} \begin{pmatrix} 1 & z_2 \\ 0 & 0 \end{pmatrix}
\]
and then compute the one variable rational inner function 
$U(z_2)$ as in \eqref{Ueq}
\[
\begin{aligned}
U(z_2) &= 
\begin{pmatrix} Q_0(z_2) & Q_1(z_2) & Q_2(z_2) \\
A_0(z_2) & A_1(z_2) & O_2 \end{pmatrix} \begin{pmatrix} I_2 & O_2 \\
O_2 & B_0(z_2) \\
O_2 & B_1(z_2) \end{pmatrix} \\
&= 
\begin{pmatrix} 0 & 0 & 0 & 1 \\
-\frac{z_2}{2} & \frac{z_2^2}{2} & \frac{1}{\sqrt{2}} & 0 \\
\frac{1}{\sqrt{2}} & \frac{z_2}{\sqrt{2}} & 0 & 0 \\
\frac{z_2}{2} & -\frac{z_2^2}{2} & \frac{1}{\sqrt{2}} & 0 \end{pmatrix}.
\end{aligned}
\]

The fourth step is to find a TFR for $U(z_2)$.
To do this we apply Remark \ref{TFRformula1}.
Let us emphasize the steps.
Divide $I_4-U(w_2)^* U(z_2)$
by $1-\bar{w}_2z_2$ and extract coefficients of $\bar{w}_2^jz_2^k$
to write
\[
\frac{I_4-U(w_2)^*U(z_2)}{1-\bar{w}_2 z_2} = 
(I_4, \bar{w}_2 I_4) Y \begin{pmatrix} I_4 \\ z_2 I_4 \end{pmatrix}
\]
where
\[
Y=\begin{pmatrix} \begin{pmatrix} \frac{1}{2} & 0 \\ 0 & 1 \end{pmatrix} & O_2 & \begin{pmatrix} 0& -\frac{1}{2} \\ 0 &0 \end{pmatrix} & O_2\\
O_2 & O_2 & O_2 & O_2 \\
\begin{pmatrix} 0 & 0 \\ -\frac{1}{2} & 0 \end{pmatrix} & O_2 &\begin{pmatrix} 0 & 0 \\ 0 & \frac{1}{2} \end{pmatrix} & O_2 \\
O_2 & O_2 & O_2 & O_2 \end{pmatrix}.
\]
Then, we factor $Y = C^*C$ where
\[
C = \left[\begin{pmatrix} \frac{1}{\sqrt{2}} & 0 \\ 0 & 1 \end{pmatrix}, O_2, \begin{pmatrix}
0 & -\frac{1}{\sqrt{2}} \\ 0 & 0 \end{pmatrix}, O_2 \right].
\]
Note that 
\[
D = \begin{pmatrix} \begin{pmatrix} \sqrt{2} & 0 \\ 0 & 1 \end{pmatrix} \\ O_2 \\ O_2 \\ O_2\end{pmatrix}
\]
is a right inverse for $C$ (i.e. $CD = I_2$).
Set 
\[
F(z_2) = C \begin{pmatrix} I_4 \\ z_2 I_4 \end{pmatrix} 
= \begin{pmatrix} \frac{1}{\sqrt{2}} & -\frac{1}{\sqrt{2}}z_2
 & 0 & 0 \\ 0 & 1 & 0 & 0 \end{pmatrix}.
\]

We need to compute the unitary (or isometry in general) $V$
such that
\[
V \begin{pmatrix} I_4 \\ z_2 F(z_2) \end{pmatrix} = \begin{pmatrix} U(z_2) \\ F(z_2) \end{pmatrix}.
\]
After equating coefficients of powers of $z_2$ this is equivalent to
\[
V \begin{pmatrix} I_4 & O_{4\times 8} \\ O_{2\times 4} & C \end{pmatrix} = 
\begin{pmatrix} U_0 & U_1 & U_2 \\ 
\left[\begin{pmatrix} \frac{1}{\sqrt{2}} & 0 \\ 0 & 1 \end{pmatrix}, O_2\right] & \left[\begin{pmatrix}
0 & -\frac{1}{\sqrt{2}} \\ 0 & 0 \end{pmatrix}, O_2 \right] & O_{2\times 4} \end{pmatrix}
\]
where $U(z_2) = U_0 + z_2 U_1 + z_2^2 U_2$.
Using the right inverse $D$ we have
\[
\begin{aligned}
V &= \begin{pmatrix} U_0 & U_1 & U_2 \\ 
\left[\begin{pmatrix} \frac{1}{\sqrt{2}} & 0 \\ 0 & 1 \end{pmatrix}, O_2\right] & \left[\begin{pmatrix}
0 & -\frac{1}{\sqrt{2}} \\ 0 & 0 \end{pmatrix}, O_2 \right] & O_{2\times 4} \end{pmatrix} 
\begin{pmatrix} I_4 & O_{4\times 2} \\ O_{8\times 4} & D\end{pmatrix} \\
&= \begin{pmatrix} 0 & 0 & 0 & 1 & 0 & 0 \\
0 & 0 & \frac{1}{\sqrt{2}} & 0 & -\frac{1}{\sqrt{2}} & 0 \\
\frac{1}{\sqrt{2}} & 0 & 0 & 0 & 0 & \frac{1}{\sqrt{2}} \\
0 & 0 & \frac{1}{\sqrt{2}} & 0 & \frac{1}{\sqrt{2}} & 0 \\
\frac{1}{\sqrt{2}} & 0 & 0 & 0 & 0 & -\frac{1}{\sqrt{2}} \\
0 & 1 & 0 & 0 & 0 & 0 \end{pmatrix}.
\end{aligned}
\]
This is the desired unitary out of which we build our TFR.
Setting $V_{11} = O_2$
\[
V_{12} = \begin{pmatrix} 0 & 1 & 0 & 0 \\ 
\frac{1}{\sqrt{2}} & 0 & -\frac{1}{\sqrt{2}} & 0 \end{pmatrix}
\]
\[
V_{21} = \begin{pmatrix} \frac{1}{\sqrt{2}} & 0 \\ 0 & 0 \\ \frac{1}{\sqrt{2}} & 0 \\ 0 & 1 \end{pmatrix} ,
V_{22} = \begin{pmatrix} 0 & 0 & 0 & \frac{1}{\sqrt{2}} \\
\frac{1}{\sqrt{2}} & 0 & \frac{1}{\sqrt{2}} & 0 \\
0 & 0 & 0 & -\frac{1}{\sqrt{2}} \\
0 & 0 & 0 & 0 \end{pmatrix}
\]
we have
\[
S(z_1,z_2) = V_{11} + V_{12} \Delta(z)( I- V_{22} \Delta(z))^{-1} V_{21}
\]
where $\Delta(z_1,z_2) = \begin{pmatrix} z_1 I_2 & O_2 \\ O_2 & z_2 I_2 \end{pmatrix}$.
This is easy to verify since $(V_{22}\Delta(z))^3 = O$ so that
the formula reduces to 
\[
S(z) = V_{12} \Delta(z)(I + V_{22} \Delta(z)+ (V_{22} \Delta(z)^2) V_{21}
\]
which can be verified by hand.

While the above method involves
several steps it is entirely systematic.
Since $S$ is a product of simple 
inner functions, there are ad hoc ways
of coming up with a TFR 
which might be shorter.

\section{Matrix Agler decompositions in two variables}\label{sec:AD}

Theorem \ref{mainthm}
makes it possible to prove
Agler's theorem (Thm \ref{Aglerthm}).
Cole-Wermer \cite{CW99}
showed that in the scalar case
it is enough to prove Agler's theorem
for rational inner functions because
holomorphic $f:\D^2\to \D$
 can be approximated locally uniformly 
by rational inner functions 
(Theorem 5.5.1 of Rudin \cite{Rudin}).
This approximation argument 
does not seem to transfer to the 
matrix-valued function setting, but there is a workaround.

\begin{lemma} \label{Schurbreakup}
Let $f:\D^d \to \C^{M\times N}$ be holomorphic and
$\|f(z)\| \leq 1$ for all $z\in \D^d$.
Suppose $\|f(z_0)\|=1$ for some $z_0 \in \D^d$.
Then, there exist unitary matrices $U_1,U_2$ such that $U_1 f U_2$ is a direct sum of a constant unitary
matrix and a matrix valued holomorphic function $g$ on $\D^d$ with $\|g(z)\|<1$ 
for all $z\in \D^d$.
\end{lemma}

\begin{proof}
If $\|f(z_0)\| = 1$, then there exists $v \in \C^N$ 
with $|v|=1$ such that $|f(z_0)v| = 1$.
By the maximum principle,
$\langle f(z)v, f(z_0)v\rangle$
is constant and equal to one.  
Then, by equality in Cauchy-Schwarz, $f(z)v \equiv f(z_0)v$.
Since $f(z)$ has at most norm one, $v$ is reducing for $f(z)$
meaning $f(z)w \perp f(z)v$ whenever $v\perp w$.  
Thus, $f(z)$ can be written in the form
\[
\begin{pmatrix} 1 & 0 \\ 0 & g(z) \end{pmatrix}
\]
using the block decomposition 
$\C f(z_0)v \oplus (f(z_0)v)^{\perp} \times 
(\C v) \oplus v^{\perp}$.
We can of course iterate this argument
until we are left with the claimed decomposition.
\end{proof}

This lets us reduce to the case
of $f$ with $\|f(z)\|<1$ for
all $z\in \D^d$. 
The following is found in Rudin's book \cite{Rudin} in the scalar case 
(see Theorem 5.5.1 of \cite{Rudin}).  Define
\[
\| f\|_{\D^d} = \sup_{z\in \D^d} \|f(z)\|.
\]

\begin{lemma} \label{Rudinapprox}
Suppose $f:\D^d \to \C^{M\times N}$ is holomorphic
and $\|f(z)\|<1$ for all $z\in \D^d$.  
Then, for any $r\in (0,1)$ and $\epsilon > 0$
there exists $P\in \C^{M\times N}[z_1,\dots,z_d]$ such that
$\|P\|_{\D^d}<1$ and $\|f-P\|_{r\D^d} < \epsilon$.

Consequently, every such $f$ is a local uniform limit
of matrix polynomials with supremum norm strictly 
less than $1$.
\end{lemma}

\begin{proof}
Set $f_r(z) = f(rz)$ for $r\in (0,1)$.
For fixed $r\in (0,1)$ there exists
$s \in (0,1)$ such that $\|f_r-f_{rs}\|_{\D^d} < \epsilon/2$
since $f_r$ is uniformly continuous on $\overline{\D}^d$.
Note $\|f_s\|_{\D^d} < 1.$
Choose a Taylor polynomial $P$ of $f_s$
such that $\|f_s - P\|_{\D^d} < \min(1-\|f_s\|_{\D^d}, \epsilon/2)$.
Then, $\|P\|_{\D^d} < 1$ and $\|f_r-P_r\|_{\D^d} \leq \|f_r-f_{rs}\|_{\D^d} + \|f_{rs} - P_r\|_{\D^d}
< \epsilon$.
\end{proof}

We need the following Fej\'er-Riesz type theorem
of Dritschel.  

\begin{theorem}[Dritschel \cite{mD1}] \label{dritschel1}
Let $T(z) = \sum_{j\in \Z^d} T_j z^j$ be a matrix-valued Laurent polynomial in $d$ variables; 
i.e. $T_j \in \C^{N\times N}$ for $j\in \Z^d$ and at most finitely many $T_j \ne 0$.
If there is a $\delta >0$ such that $T(z) \geq \delta I$ on $\T^d$, then there exists a matrix polynomial $A\in \C^{M\times N}[z_1,\dots,z_d]$
such that $T = A^*A$ on $\T^d$.
\end{theorem}

We sketch a simple proof with some new elements in the appendix; see Subsection \ref{Fejerproofs}.

\begin{lemma} \label{polytfr}
If $P:\D^d \to \C^{M\times N}$ is a matrix polynomial
such that $\|P\|_{\D^d}<1$ then there exists
a matrix polynomial $A$ such that 
$\begin{pmatrix} P \\ A\end{pmatrix}$ 
is iso-inner.
If $d=1,2$, then $P$ has a 
finite contractive TFR.
\end{lemma}

\begin{proof}
On $\T^d$, $I-P^*P$ is a positive definite
matrix Laurent polynomial.  
By Theorem \ref{dritschel1} we
can factor $I-P^*P = A^*A$.  
Then, $S = \begin{pmatrix} P \\ A \end{pmatrix}$
is isometry-valued on $\T^d$.
If $d=1,2$, then $S$ has a finite isometric TFR
by Theorem \ref{mainthm}
and hence $P$ possesses a finite
contractive TFR by Proposition \ref{partthm}.
\end{proof}

Positive semi-definite kernels are defined in Definition \ref{def:PSD}.
Notice that an expression of the form $F(w)^*F(z)$ will
always be positive semi-definite.  By the above
lemma and Theorem \ref{equivalences}, any matrix
polynomial $P\in \C^{M\times N}[z_1,z_2]$ with $\|P\|_{\D^2}<1$
will satisfy a formula of the form
\[
I-P(w)^*P(z) = k_0(w,z) + \sum_{j=1}^{2} (1-\bar{w}_j z_j) k_j(w,z)
\]
where $k_0,k_1,k_2$ are positive semi-definite kernels.
The term $k_0$ can be absorbed into $k_1$ since 
\[
\frac{k_0(w,z)}{1-\bar{w}_1 z_1}
\]
is positive semi-definite by the Schur product theorem.
Thus, the following corollary holds for such strictly contractive
matrix polynomials in two variables.  Such formulas
are called \emph{Agler decompositions}.

\begin{corollary}\label{maincor}
Let $f:\D^2 \to \C^{M\times N}$ be holomorphic with $\|f(z)\|\leq 1$ for $z\in \D^2$.  Then, there exist
positive semi-definite kernels $k_1,k_2:\D^2\times \D^2 \to \C^{N\times N}$ such that
\[
I-f(w)^*f(z) = \sum_{j=1}^{2} (1-\bar{w}_j z_j) k_j(w,z).
\]
\end{corollary}

\begin{proof}[Sketch of Proof]
The hard work has already been
done while the general outline
and some technicalities are essentially 
in \cite{CW99} so we only sketch the proof.
We can assume that $f$ is point-wise strictly contractive by Lemma 
\ref{Schurbreakup}.
Then, $f$ is a local uniform limit of matrix
polynomials with supremum norm strictly less than one by Lemma
\ref{Rudinapprox}.
Each of these possesses an Agler decomposition by
the discussion above.

The final part of the argument is the piece found in \cite{CW99}.
The kernels in the Agler decomposition are locally bounded because of the
estimate
\[
\frac{1}{(1-|z_1|^2)(1-|z_2|^2)} I \geq \frac{I-f(z)^*f(z)}{(1-|z_1|^2)(1-|z_2|^2)}
\geq \frac{k_1(z,z)}{1-|z_2|^2} \geq k_1(z,z).
\]
This shows the kernels in Agler decompositions form
a normal family.
Subsequences converge locally uniformly
to form positive semi-definite kernels in an Agler decomposition for $f$.
\end{proof}

The above corollary proves Theorem \ref{Aglerthm}.
The proof is essentially the same as $(3)\implies (1)$
in the equivalences theorem (Thm \ref{equivalences})
since positive semi-definite kernels can be factored as $F(w)^*F(z)$
for some possibly operator valued function $F$.
Readers who have ventured this far 
(and are not in the cognoscenti of
this material) may benefit from
some context at this point.
The fundamental contribution of Agler
can perhaps be encapsulated in 
the following result.

\begin{theorem}[Agler \cites{Agler1, Agler2}] 
\label{Aglerthmlong}
Let $f:\D^d \to \C^{M\times N}$ be holomorphic.
Assume $\|f(z)\|\leq 1$ for $z\in \D^d$.
Then, the following are equivalent.
\begin{enumerate}
\item $f$ satisfies a von Neumann inequality:
\[
\|f(T)\| = \| (f_{j,k}(T))_{j,k}\| \leq 1
\]
for every 
$d$-tuple $T=(T_1,\dots,T_d)$ of pairwise
commuting strictly contractive operators (on
some underlying Hilbert space);

\item $f$ has an Agler decomposition:
there exist positive semi-definite kernels 
$k_1, \dots, k_d:\D^d\times \D^d \to \C^{N\times N}$
such that
\[
I-f(w)^*f(z) = \sum_{j=1}^{d} (1-\bar{w}_j z_j) k_j(w,z);
\]
\item $f$ has a contractive transfer function 
realization:
there exists a contractive operator with block
decomposition
$T = \begin{pmatrix} A& B \\ C & D\end{pmatrix}$
on some Hilbert space
such that
\[
f(z) = A + B \Delta(z) (I-D\Delta(z))^{-1} C
\]
where $\Delta(z) = \sum_{j=1}^{d} z_j P_j$ and the $P_j$ 
are pairwise orthogonal orthogonal projections which
sum to the identity on the domain of $D$.
\end{enumerate}
\end{theorem}

Theorem \ref{Aglerthm} was originally
proven via And\^{o}'s inequality \cite{Ando} 
which gives item (1)
above.  The approach we have given sidesteps 
the use of von Neumann's inequality and
the implication $(1)\implies (2)$ in Theorem
\ref{Aglerthmlong}.  The proof of $(1)\implies (2)$
is possibly the hardest part of the theorem
and is non-constructive as it uses
a Hahn-Banach cone separation argument.
On the other hand, $(2)\implies (1)$
is a relatively straightforward matter
 of ``plugging'' the $d$-tuple $T$ into
 the Agler decomposition in item (2)
 in an appropriate sense.
 See \cite{CW99} for details.
Ball-Sadosky-Vinnikov \cite{BSV05}
have a different way to prove
Theorem \ref{Aglerthm} directly
using multi-evolution scattering systems.
Theorem \ref{Aglerthm}'s analogue for $3$
or more variables fails because the von Neumann
inequality fails for 3 or more contractions \cite{Varo}.
Thus, Theorem \ref{Aglerthmlong} gives
the best way of demonstrating that a function
does not have a contractive TFR; namely, showing
that it fails the von Neumann inequality.  
It is probably difficult to directly show that a function
fails item (2) or (3) in Theorem \ref{Aglerthmlong}.

We conclude this section by plugging Dritschel's strong Fej\'er-Riesz
type result (stated below) into earlier arguments in order 
to show rational contractive matrix-valued functions in two variables
have a finite contractive TFR (Theorem \ref{finitetfr}).

\begin{theorem}[Dritschel \cite{mD2}] \label{dritschel2}
Let $T(z) = \sum_{j\in \Z^2} T_j z^j$ be a matrix-valued Laurent polynomial in two variables; 
i.e. $T_j \in \C^{N\times N}$ for $j\in \Z^2$ and at most finitely many $T_j \ne 0$.
If $T(z) \geq 0$ on $\T^2$, then there exists a matrix polynomial $A\in \C^{M\times N}[z_1,z_2]$
such that $T = A^*A$ on $\T^2$.
\end{theorem}

This theorem is considerably deeper than Theorem \ref{dritschel1}, 
and both theorems also apply
to operator-valued functions.
An earlier sums of squares theorem of Scheiderer, which 
applied to polynomials on a much more general class of
two dimensional domains (than simply $\T^2$),
implies 
Theorem \ref{dritschel2} in the scalar case \cite{Scheiderer}.

\begin{proof}[Proof of Theorem \ref{finitetfr}]
Apply the proof of Proposition \ref{onevarfinitetfr}
with Theorem \ref{dritschel2} in place of Theorem \ref{genFR}.
\end{proof}

\section{More on finite TFRs}\label{moretfr}
We need to collect one more fact about finite-dimensional TFRs 
before proving the minimality theorem.
If we have
an Agler decomposition of an iso-inner
function $S=Q/p$
written in lowest terms, then the 
sums of squares terms are rational 
with denominator $p$.

\begin{theorem}\label{polydecomps}
Suppose $S:\D^d\to \C^{M\times N}$ is
rational and iso-inner.  
Write $S = Q/p$ in lowest terms with
 $Q \in \C^{M\times N} [z_1,\dots,z_d]$ and 
$p\in \C[z_1,\dots, z_d]$.  
Suppose we have an Agler decomposition
\[
I_N - S(w)^*S(z) = \sum_{j=1}^{d} (1-\bar{w}_j z_j) F_j(w)^*F_j(z)
\]
where the $F_j$ are matrix functions.
Then, for $j=1,\dots, d$, $p(z)F_j(z)$ 
is a matrix polynomial.
\end{theorem}

The significance of this theorem
is that although $S$ has a TFR
with denominator $\det(I-D\Delta(z))$,
this polynomial may not be the lowest
degree denominator of $S$.  

\begin{proof}
By Theorem \ref{equivalences} we already
see that each $F_j$ is rational and
holomorphic in $\D^d$.  To prove that $H_j := p F_j$
is a matrix polynomial consider
\[
\overline{p(w)}p(z) I_N- Q(w)^* Q(z) = \sum_{j=1}^{d} (1-\bar{w}_j z_j) H_j(w)^* H_j(z).
\]
Fix $\tau \in \T^d$ and set $z = \zeta \tau, w = \eta \tau$ for $\zeta, \eta \in \D$.
Then
\[
\overline{p(\eta \tau)} p(\zeta \tau) I_N - Q(\eta \tau)^* Q(\zeta \tau)
 = (1-\bar{\eta} \zeta) \sum_{j=1}^{d} H_j(\eta \tau)^* H_j(\zeta \tau).
 \]
 Because $S^*S = I_N$ on $\T^d$, the left hand side
 above is divisible by $(1-\bar{\eta} \zeta)$ and therefore
 \[
 \sum_{j=1}^{d} H_j(\eta \tau)^* H_j(\zeta \tau)
 \]
 is a polynomial in $\zeta, \bar{\eta}$ of degree in each less than the total degree
 of $p$ and $Q$.
 For simplicity we can regroup $\sum_{j=1}^d H_j(w)^* H_j(z) = H(w)^*H(z)$
 where now $H(\eta \tau)^* H(\zeta \tau)$ is a polynomial in $\zeta,\bar{\eta}$ for
 every $\tau\in \T^d$.  If we write out the homogeneous expansion of $H$,
 \[
 H(z) = \sum_{j=0}^{\infty} P_j(z)
 \]
 we see that 
 \[
 H(\eta \tau)^* H(\zeta \tau) = \sum_{j,k} \bar{\eta}^j \zeta^k P_j(\tau)^* P_k(\tau).
 \]
 In particular, 
 for $j$ greater than the total degrees
 of $p$ and $Q$, the coefficient of  $\bar{\eta}^j\zeta^j$
 vanishes for every $\tau$;
 namely, we have $P_j(\tau)^*P_j(\tau) \equiv 0$ for all $\tau \in \T^d$.  
 Since $P_j$ is a matrix polynomial, this implies $P_j \equiv 0$
 for $j$ greater than the total degrees of $p$ and $Q$.
 Therefore, $H$ is a polynomial.
 \end{proof}

We conclude this short section with a few asides.
The Agler norm (sometimes Schur-Agler norm)
for holomorphic $f:\D^d\to \C^{M\times N}$ is
\begin{equation} \label{Aglernorm}
\|f\|_{\mcA_d} := \sup_{T} \|f(T)\|
\end{equation}
where the supremum is taken over all
$d$-tuples $T=(T_1,\dots,T_d)$ of strictly contractive 
pairwise commuting operators on some Hilbert space.
The Agler class $\mcA_d$ consists of
 functions satisfying $\|f\|_{\mcA_d} \leq 1$. 

The argument in the proof above is 
related to the argument used to prove the
following automatic finite-dimensionality result.

\begin{theorem} \label{CWthm}
Suppose $S:\D^d\to \C^{M\times N}$
is rational, iso-inner or coiso-inner ($SS^*=I$ on $\T^d$),
and belongs to the Agler class $\mcA_d$.
Then, $S$ has a finite-dimensional isometric (resp. coisometric) TFR as in 
Theorem \ref{equivalences}.
\end{theorem}

The essence of this theorem was first proved 
in Cole-Wermer \cite{CW99}.
Although it was only stated and proved
in the scalar case for $d=2$, the proof goes through
easily to all $d$ and for iso-inner functions.  
We gave a proof with some bounds
on degrees and the numbers of squares involved
in the scalar case in \cite{KneseRIFITSAC}.
A proof of the square
matrix-valued case is in \cite{BallKal}.
Extending to the iso-inner (non-square) case
causes no difficulties.
The coisometric 
case follows from Proposition \ref{breve}.
A proof where $S$
is assumed to be a polynomial 
is also given in \cite{wavelet}. 
The next theorem also produces
a family of functions with finite TFRs.

\begin{theorem}[Grinshpan et al \cite{Getal}]\label{Getalthm}
Suppose $S: \D^d \to \C^{M\times N}$ 
is rational, analytic on a neighborhood 
of $\overline{\D}^d$, and $\|S\|_{\mcA_d} < 1$.
Then, $S$ has a finite-dimensional contractive
TFR as in Theorem \ref{equivalences}.
\end{theorem}

The following question asks about what is 
still left open.  

\begin{question}
For $d>2$,
if $S:\D^d \to \C^{M\times N}$  is rational, $\|S\|_{\mcA_d} = 1$, 
and is neither iso-inner nor coiso-inner, then does $S$ 
have a finite-dimensional contractive TFR?
\end{question}

We also do not know how essential analyticity on $\overline{\D}^d$ is for Theorem \ref{Getalthm}.
Note $d=1,2$ follows from Theorem \ref{finitetfr}.

\section{Kummert's minimality theorem}

In this section we discuss minimality of size breakdowns
for finite TFRs, namely Theorem \ref{minthm}.  
Minimality in one variable follows directly from 
Theorem \ref{equivalences}.

\begin{prop}
Let $S:\D\to \C^{M \times N}$ be rational and iso-inner.
Then, the minimal size of an isometric TFR for $S$ is the rank of the 
positive semi-definite kernel
\[
(w,z) \mapsto \frac{I - S(w)^* S(z)}{1-\bar{w}z}.
\]
\end{prop}

The definition of the rank of a positive semi-definite kernel is 
given in Definition \ref{rank} in the Appendix.
In two variables, we will frequently refer to the dominant $z_1$-term $G^*G$
and sub-dominant $z_2$-term $H^*H$ 
associated to $S$ which were constructed in the proof
of Theorem \ref{mainthm}; see the end of Section \ref{sec:twovar}.
Note that
the number of rows of $G$ matches the
generic rank of the matrix $T(z_2)$ 
as in equation \eqref{paraSC}.
This cannot be reduced
because this is the generic or maximal rank of 
the positive semi-definite kernels
\[
(w_1,z_1) \mapsto \frac{I-S(w_1,z_2)^*S(z_1,z_2)}{1-\bar{w}_1z_1} \text{ as } z_2 \text{ varies over } \T.
\]
  Note division of \eqref{paraSC} by $\overline{p(w_1,z_2)} p(z_1,z_2)$
will not change the rank of the positive semi-definite kernel
and does not introduce any poles in $\D$ since $p(\cdot,z_2)$ has
no zeros in $\D$ by Lemma \ref{Hurwitzlemma}.

We claim that in the inner case 
the rank of $H^*H$ is also as small as 
possible.  We suspect this happens in the iso-inner
case but cannot prove it.

\begin{question}
If $S:\D^2\to \C^{M\times N}$ is iso-inner 
(and not inner), does the construction in Section \ref{sec:twovar}
produce a size breakdown $(r_1,r_2)$
with $r_1$ equal to the generic size  
of a TFR for $S(\cdot,z_2)$ (for $z_2\in \T$)
and $r_2$ equal to the generic size of a TFR
for $S(z_1,\cdot)$ (for $z_1 \in \T$)?
\end{question}

This question is subtle because every iso-inner
function $S$ is a submatrix of an inner function $\Phi$
with the same size breakdown.  
We have built a size breakdown with 
$r_1$ minimal so $r_1$ must also be minimal for $\Phi$.
We could then build a TFR with size breakdown $(r_1,r_2^*)$
where $r_2^*$ is minimal for $\Phi$.  Is it minimal for the
restriction to $S$?

The next result characterizes $G^*G$ and $H^*H$.

\begin{prop} \label{domination}
Assume $S:\D^2\to \C^{M \times N}$ is rational and iso-inner.
Write $S=Q/p$ in lowest terms.
Suppose we had a formula
\[
\overline{p(w)}p(z)I - Q(w)^*Q(z) = (1-\bar{w}_1 z_1) \Gamma_1(w)^*\Gamma_1(z) + 
(1-\overline{w}_2 z_2) \Gamma_2(w)^* \Gamma_2(z)
\]
where $\Gamma_1, \Gamma_2$ are matrix polynomials.
Then, 
\begin{equation} \label{Gdominant}
(w,z) \mapsto \frac{G(w)^*G(z) - \Gamma_1(w)^* \Gamma_1(z)}{1-\bar{w}_2 z_2}  =
\frac{\Gamma_2(w)^* \Gamma_2(z) - H(w)^* H(z)}{1-\bar{w}_1 z_1}
\end{equation}
is a positive semi-definite polynomial kernel.
Here again $G^*G$ is the dominant $z_1$-term and
$H^*H$ is the sub-dominant $z_2$-term.
\end{prop}

This result characterizes $G^*G$ as maximal 
and $H^*H$ as minimal in the above sense.  
Indeed, if some other kernel $L^*L$ satisfied 
the same property as $G^*G$ then 
both
\[
\frac{G^*G-L^*L}{1-\bar{w}_2 z_2} \text{ and } \frac{L^*L- G^*G}{1-\bar{w}_2z_2}
\]
would be positive semi-definite forcing $G^*G=L^*L$.

\begin{proof}[Proof of Proposition \ref{domination}]
If we set $z_2=w_2\in \T$ we get
\[
\frac{\overline{p(w)} p(z) I - Q(w)^* Q(z)}{1-\bar{w}_1 z_1} = \Gamma_1(w)^*\Gamma_1(z) = G(w)^*G(z).
\]
The left side has degree at most $n_1-1$ in $z_1$.
We claim $\Gamma_1(z)$ has
degree at most $n_1-1$ in $z_1$.
Consider $\Gamma_1$'s top degree term $\gamma(z_2) z_1^k$
where $\gamma(z_2)$ is a matrix polynomial.
Then, the term $\bar{w}_1^k z_1^k$ appears on the right hand side
with coefficient $\gamma(z_2)^*\gamma(z_2)$ for $z_2 \in \T$.  
If $k>n_1-1$ then $\gamma(z_2)^*\gamma(z_2) \equiv 0$ on $\T$ implying
$\gamma(z_2) \equiv 0$ on $\T$ and also on $\C$ by analyticity.
Thus, $\Gamma_1$ has degree at most $n_1-1$ in $z_1$.

Just as we have factored $G(z) = A(z_2)\Lambda(z_1)$ 
we can also factor $\Gamma_1(z) = C(z_2) \Lambda(z_1)$. 
Recall $\Lambda(z_1) = (I, z_1 I, \cdots, z_1^{n_1-1} I)^t$.
Upon extracting coefficients of $\bar{w}_1^jz_1^k$ we see that
\[
A(z_2)^* A(z_2) = C(z_2)^*C(z_2) 
\]
for $z_2 \in \T$. 
This is related to characterizing uniqueness
in the matrix Fej\'er-Riesz theorem.  
We address this in the appendix in Theorem \ref{genFR-uniqueness}.
By Theorem \ref{genFR-uniqueness}, since $A$ has a left inverse,
there exists a one variable iso-inner function $\Phi$
 such that $C = \Phi A$.  

So,
\[
\frac{A(w_2)^*A(z_2) - C(w_2)^* C(z_2)}{1-\bar{w}_2 z_2}
=A(w_2)^* 
\left(\frac{I-\Phi(w_2)^*\Phi(z_2)}{1-\bar{w}_2 z_2}\right)
A(z_2)
 \]
which is positive semi-definite.  Applying $\Lambda(w_1)^*$ on the left and $\Lambda(z_1)$
on the right we get 
\[
\frac{G(w)^* G(z) - \Gamma_1(w)^* \Gamma_1(z)}{1-\bar{w}_2 z_2} = \frac{\Gamma_2(w)^*\Gamma_2(z) - H(w)^*H(z)}{1-\bar{w}_1 z_1}
\]
is positive semi-definite.  
It is a polynomial kernel because
$A^*A=C^*C$ on $\T$.
\end{proof}

We now switch to the square/inner case and
show that the Kummert construction
gives the best possible size breakdown $r=(r_1,r_2)$.
We need to show $H(w)^*H(z)$ has the minimal rank possible
in the sense that it matches the generic size of a 
TFR for $S(z_1,\cdot)$ for $z_1\in \T$.
To do this, we show that we can ``reflect'' an Agler decomposition
of $S$ to get an Agler decomposition for $\breve{S}$
and this reflection reverses the dominant and sub-dominant 
properties of $G^*G$ and $H^*H$.
This is not the original approach of Kummert;
instead it more closely resembles the Hilbert space 
approach in \cite{BickelKnese}.  Recall $\breve{S}(z) = S(\bar{z})^*$.

\begin{prop} \label{reverseprop}
Suppose $S:\D^2\to \C^{N\times N}$ is rational and inner.
Write $S=Q/p$ in lowest terms.
Suppose we had a formula
\begin{equation} \label{ourformula}
\overline{p(w)}p(z)I_N - Q(w)^*Q(z) = (1-\bar{w}_1 z_1) \Gamma_1(w)^*\Gamma_1(z) + 
(1-\overline{w}_2 z_2) \Gamma_2(w)^* \Gamma_2(z)
\end{equation}
where $\Gamma_1, \Gamma_2$ are matrix polynomials.
Then,
\begin{equation} \label{reflectops}
\tilde{\Gamma}_1(z) := \frac{1}{z_1p(1/z)} \Gamma_1(1/z)\breve{S}(z)
\text{ and }
\tilde{\Gamma}_2(z) := \frac{1}{z_2 p(1/z)} \Gamma_2(1/z) \breve{S}(z)
\end{equation}
are matrix polynomials and
\begin{equation} \label{desiredformula}
\overline{\breve{p}(w)} \breve{p}(z)I-\breve{Q}(w)^* \breve{Q}(z) 
= (1-\bar{w}_1 z_1) \tilde{\Gamma}_1(w)^*\tilde{\Gamma}_1(z) + 
(1-\bar{w}_2 z_2) \tilde{\Gamma}_2(w)^*\tilde{\Gamma}_2(z).
\end{equation}
The sub-dominant $z_2$-term of $S$ reflects to the
dominant $z_2$-term of $\breve{S}$.
\end{prop}

When we say \emph{reflects} above we mean the operations:
\begin{equation} \label{reflects}
\Gamma_1 \mapsto \tilde{\Gamma}_1 \text{ and } \Gamma_2 \mapsto \tilde{\Gamma}_2
\end{equation}
listed in the proposition statement equation \eqref{reflectops}. 
Notice that reflection of the $\Gamma_1$ term is slightly different
from the reflection of the $\Gamma_2$ term.

\begin{proof}[Proof of Proposition \ref{reverseprop}]
Since $S(z)^*S(z) = I$ on $\T^2$ (where defined) we
have $I=S(1/\bar{z})^*S(z) = S(z) S(1/\bar{z})^*$
 for $z\in \C^2$ where defined.  
(This is where $M=N$ gets used.)
 So, $Q(1/z)\breve{Q}(z) = p(1/z) \breve{p}(z)I$.
Now, take equation \eqref{ourformula}, replace $z,w$ with $1/z,1/w$, 
multiply on the right by $\breve{Q}(z)$ and left by $\breve{Q}(w)^*$,
and finally divide through by $-\overline{p(1/w)}p(1/z)$
to get \eqref{desiredformula}
after applying various simplifications.
Of course, we have the caveat that the formula
only holds where all of the operations are defined.
Fortunately, \eqref{desiredformula} only needs to hold on an open set
for the proof of (3)$\implies$ (1),(2) in Theorem \ref{equivalences}
to go through (bonus (B1) of Theorem \ref{equivalences} addresses this).  
We automatically obtain that $\tilde{\Gamma}_1, \tilde{\Gamma}_2$
are polynomials by Theorem \ref{polydecomps}, since
if $Q/p$ is in lowest terms then $\breve{Q}/\breve{p}$ is too.

If we reflect equation \eqref{Gdominant}
in the sense of replacing $z,w$ with $1/z,1/w$ and
conjugating by $\breve{Q}$ 
we obtain
\[
\bar{w}_1z_1 \frac{\tilde{G}(w)^*\tilde{G}(z) - \tilde{\Gamma}_1(w)^* \tilde{\Gamma}_1(z)}{1-(\bar{w}_2z_2)^{-1}} =
\bar{w}_2z_2 \frac{\tilde{\Gamma}_2(w)^*\tilde{\Gamma}_2(z) - \tilde{H}(w)^* \tilde{H}(z)}{1-(\bar{w}_1 z_1)^{-1}} 
\]
which rearranges into
\[
\frac{\tilde{\Gamma}_1(w)^*\tilde{\Gamma}_1(z) - \tilde{G}(w)^* \tilde{G}(z)}{1-\bar{w}_2z_2} =
\frac{\tilde{H}(w)^*\tilde{H}(z) - \tilde{\Gamma}_2(w)^* \tilde{\Gamma}_2(z)}{1-\bar{w}_1 z_1}. 
\]
This is still a positive semi-definite polynomial kernel.  Thus,
$\tilde{H}^*\tilde{H}$ dominates an arbitrary $z_2$-term
making it the dominant $z_2$-term for $\breve{S}$.
\end{proof}

\begin{proof}[Proof of Theorem \ref{minthm}]
By Proposition \ref{reverseprop} 
the subdominant $z_2$-term $H^*H$ of $S$ reflects to 
the dominant $z_2$-term of $\breve{S}$, $\tilde{H}^*\tilde{H}$.
Note that this reflection does not change the rank of 
a positive semi-definite kernel.
The rank of $\tilde{H}^*\tilde{H}$ is then the generic rank of
\[
(w_2,z_2) \mapsto \frac{I-\breve{S}(z_1,w_2)^* \breve{S}(z_1,z_2)}{1-\bar{w}_2 z_2}
\]
for $z_1\in \T$.  This matches the generic
size of a TFR for $\breve{S}(z_1,\cdot)$ which
matches the generic size of a TFR for $S(z_1,\cdot)$
by the adjunction formula, Proposition \ref{breve}.  
Thus the rank of $H^*H$ matches the generic rank of
\[
(w_2,z_2) \mapsto \frac{I-S(z_1,w_2)^* S(z_1,z_2)}{1-\bar{w}_2 z_2}.
\]
\end{proof}

\section{Application to inner polynomials}
Of special interest in the papers connecting
wavelets to TFRs
is the case of iso-inner 
and inner polynomials \cites{wavelet, CCCP}.  
In one variable, we have the following well-known
result.

\begin{prop}\label{polyprop}
Let $S \in \C^{M\times N}[z]$ be iso-inner.  Then,
every isometric TFR of minimal size for $S$ is built out of
an isometric matrix
$T = \begin{pmatrix} A & B \\ C & D \end{pmatrix}$ 
where $D$ is nilpotent.
\end{prop}

We prove this using the following also well-known characterization of
minimality.  

\begin{prop}\label{spanprop}
Let $S: \D\to \C^{M\times N}$ be rational and iso-inner
with minimal isometric  TFR built out of the isometric matrix
$T = \begin{pmatrix} A & B \\ C & D \end{pmatrix}$.
Then,
\[
\text{span} \{ \text{range} (D^jC) : j=0,1,\dots\} = \text{domain} (D) \text{ and } \bigcap_{j\geq 0} \text{kernel} (BD^j) = \{0\}
\]
\end{prop}

\begin{proof}
First note that if $S$ has a TFR via $T$, meaning 
$S(z) = A+zB(I-zD)^{-1} C$,
 then
it also has a TFR via
\[
\begin{pmatrix} I & 0 \\ 0 & U^*\end{pmatrix} T \begin{pmatrix} I & 0 \\ 0 & U\end{pmatrix}
=
\begin{pmatrix} A & BU \\ 
U^*C & U^*DU \end{pmatrix}
\]
where $U$ is a unitary matrix with the same dimensions as $D$.
This is apparent from the formula 
$A + z BU (I-zU^*DU)^{-1} U^*C = S(z).$
We can apply a unitary change of coordinates and
 break up the domain/codomain of $D$ into
 $\mcH = \text{span} \{ D^jC : j=0,1,\dots\}$
 and its orthogonal complement $\mcH^{\perp}$.
In these new coordinates $T$ takes the form
\[
\begin{matrix} & \begin{matrix} \C^N & \mcH & \mcH^{\perp} \end{matrix} \\
\begin{matrix} \C^M \\ \mcH \\ \mcH^{\perp}\end{matrix} &
\begin{pmatrix} A & B_1 & B_2 \\ C & D|_{\mcH} &  * \\ 0 & 0 & * \end{pmatrix} \end{matrix}.
\]
since $D$ maps $\mcH$ to itself and $\text{range}(C) \subset \mcH$.
Since the formula for $S$ is only determined by $D|_{\mcH}$,
we see that $S$ has an isometric TFR via the matrix
$\begin{pmatrix} A & B_1 \\ C & D|_{\mcH} \end{pmatrix}$
which has a smaller size unless $\mcH^{\perp} = \{0\}$ or rather $\mcH = \text{domain} (D)$.

For the second identity, we break up the domain of $D$ into
$\mathcal{L} = \bigcap_{j\geq 0} \text{kernel}(BD^j)$ 
and its orthogonal complement $\mcL^{\perp}$.
Using this orthogonal decomposition we can write
$T$ in new coordinates as
\[
\begin{matrix} & \begin{matrix} \C^N & \mathcal{L}^{\perp} & \mathcal{L} \end{matrix}\\
\begin{matrix} \C^M \\ \mathcal{L}^{\perp} \\ \mathcal{L} \end{matrix} &
\begin{pmatrix} A & B & 0 \\
C_1 & D_{11} & 0 \\ 
C_2 & D_{21} & D|_{\mathcal{L}} \end{pmatrix} \end{matrix}
\]
since $B$ maps $\mathcal{L}$ to $0$ while 
$D$ maps $\mathcal{L}$ into itself.  But since
this is an isometry we must have $D|_{\mathcal{L}}$ 
a unitary which forces $C_2,D_{21} = 0$.
This means $S$ is given by the TFR with 
isometry
$\begin{pmatrix} A & B \\ C_1 & D_{11} \end{pmatrix}.$
This has smaller size unless $\mathcal{L} = \{0\}$.
\end{proof}

\begin{proof}[Proof of Proposition \ref{polyprop}]
If $S(z) = A + zB(I-zD)^{-1}C$ is
a polynomial, then necessarily $BD^jC = 0$ for 
all $j$ large enough.  By Proposition \ref{spanprop},
$BD^n = 0$ for $n$ large enough.  
Then, 
\[
\text{range} (D^n) \subset \bigcap_{j\geq 0} \text{kernel}(BD^j)
\]
implying $\text{range}(D^n) = 0$ or rather $D^n=0$.
\end{proof}

Minimality of TFR representations
in the rational inner case in two variables makes it
possible to prove an analogous result for
inner matrix-valued polynomials in two variables.
Our approach uses determinants to count
the size of minimal TFRs.  The following is a standard result 
in one variable.  We provide a proof in Subsection \ref{PSDsec}.

\begin{prop} \label{degdet}
Let $S: \D \to \C^{N\times N}$ be a rational inner function.  
Then, $\deg \det S$ equals the size of a minimal TFR for $S$.
\end{prop}

Since $S$ is rational inner, $\det S$ is a scalar rational inner function
in one variable which is a finite Blaschke product.  So, 
the $\deg \det S$ refers to the degree of the numerator 
when written in lowest terms.
This immediately yields a method using determinants
to calculate the optimal size breakdown for
rational inner functions in two variables.
(This is another place where it helps to have square matrices.)

\begin{theorem}[Kummert]
If $S: \D^2 \to \C^{N\times N}$ is rational inner, 
then the minimal size breakdown $r=(r_1,r_2)$ of a TFR
for $S$ is
\[
r_j = \deg_j \det S(z_1,z_2) \text{ for } j=1,2.
\]
\end{theorem}

Similarly, for all but finitely many $\zeta \in \T$, the degree of 
\[
z \mapsto \det S(z, \zeta z) 
\]
is $r_1+r_2$.
Therefore, the generic size of a TFR for $z\mapsto S(z,\zeta z)$
is $r_1+r_2$.  This shows that generic restrictions 
to slices
of our two variable minimal TFRs yield
minimal TFRs for restricted functions.

\begin{proof}[Proof of Theorem \ref{thmpolyinner}]
The above argument shows that if a polynomial
inner function $S$ has a minimal 
TFR via the unitary $U=\begin{pmatrix} A & B \\ C & D\end{pmatrix}$
and projections $P_1,P_2$ as in Theorem \ref{mainthm}
then $z \mapsto S(z,\zeta z)$ has
minimal unitary TFR via the unitary
\[
\begin{pmatrix} A & B \\ C & D \end{pmatrix} \begin{pmatrix} I & 0 \\ 0 & P_1+\zeta P_2 \end{pmatrix}.
\]
By Proposition \ref{polyprop}, $D \Delta(1,\zeta)$
is nilpotent for all but finitely many $\zeta \in \T$.
This means $(D\Delta(1,\zeta))^N = 0$
for all but finitely many $\zeta \in \T$.
Since this is a polynomial equation we have 
$(D\Delta(1,\zeta))^N \equiv 0$ and since $D \Delta(z)$
is homogeneous we also have $(D \Delta(z_1,z_2))^N \equiv 0$.
Thus, $D\Delta(z)$ is always nilpotent.
\end{proof}

This leads to the interesting question
of describing contractions $D$ such that $D \Delta(z)$ is nilpotent for all $z$.
An easy way to produce examples would be
to make $D$ strictly upper triangular and 
choose the projections $P_1,P_2$ 
via projections onto the span of subsets
of standard basis vectors.
For such examples, $D\Delta(z)$ is
triangular; however, it is possible to
produce matrices $D_1,D_2$ such
that $z_1 D_1+z_2D_2$ is nilpotent
for all $z$ yet is not triangularizable independent
of $z$; see \cite{nilpotent}.
This could be an interesting source of
examples.

\section{Appendix: auxiliary results}

\subsection{Maximum principle for rational iso-inner functions}

\begin{prop}\label{maxprinciple}
Suppose $S:\D^d \to \C^{M\times N}$ 
is rational, analytic in $\D^d$, and 
$\|S(z)\| \leq 1$ for $z\in \T^d$ where defined.
Then, $\|S(z)\|\leq 1$ for all $z\in \D^d$.
\end{prop}

Rationality is a key assumption since 
$f(z) = \exp\left(\frac{1+z}{1-z}\right)$ 
is unimodular on $\T\setminus \{1\}$
and analytic on $\C \setminus \{1\}$ yet
not bounded by $1$ in $\D$. 

\begin{proof}
We can reduce to the scalar
case by considering arbitrary
unit vectors $v,w$ and the function
$F(z) = w^*S(z)v$.   
Fix $\omega \in \T^d$ and consider the one 
variable rational function $f(\zeta) = F(\zeta \omega).$
This function is bounded by $1$ on $\T$
away from its potential finite number of poles.
But, $f$ must be unbounded near a pole, 
so any singularities on the boundary are
removable.  Hence, $f$ is analytic
on $\overline{\D}$ and bounded by $1$
by the maximum principle.
This implies $F$ is bounded by $1$
at any point of $r\T^d$ for $r<1$.  
Given any $z \in \D^d$, we
can calculate $F(z)$
as a Poisson integral of $F$ on $r \T^d$
for $\|z\|_{\infty} < r < 1$ to
see that $|F(z)| \leq 1$.
\end{proof}

\subsection{Fej\'er-Riesz proofs} \label{Fejerproofs}
A more traditional and well-known version of the
matrix Fej\'er-Riesz theorem is as follows.  See \cite{DGK} for a proof.

\begin{theorem}\label{FR}
Let $T(z) = \sum_{j=-n}^{n} T_j z^j$ be a matrix Laurent polynomial ($T_j \in \C^{N\times N}$)
such that $T(z) \geq 0$ for $z \in \T$ and $\det T(z)$ is not identically zero.  

Then, there exists a matrix polynomial $A \in \C^{N\times N}[z]$ of degree at most $n$ such that
$T = A^* A$ on $\T$
and $\det A(z) \ne 0$ for $z \in \D$.
\end{theorem}

We think it is worthwhile to show how to go from this theorem 
to the degenerate version, Theorem \ref{genFR}, using
ideas from \cite{Dj}.
The key tool is the Smith normal form.  

\begin{theorem}[Smith normal form] \label{SNF}
Let $P \in \C^{M\times N}[z]$ be a matrix polynomial.
Then, there exist $T_1\in \C^{M\times M}[z],T_2\in \C^{N\times N}[z]$ with matrix
polynomial inverses (equivalently, with constant determinants)
and $D \in \C^{M\times N}[z]$ 
such that $P = T_1 D T_2$.
The matrix $D$ has the following form:
every entry off the
main diagonal of $D$ is zero and the main diagonal consists of polynomials $d_1,\dots, d_k$
such that $d_j$ divides $d_{j+1}$. 
Here $k= \min\{N,M\}$ and 
the $d_j$ may be zero for $j$ large enough.
\end{theorem}

See Hoffman-Kunze \cite{HK}. 

\begin{proof}[Proof of Theorem \ref{genFR}]
The function $G(z) = z^n T(z)$ is a polynomial matrix
and therefore has Smith normal form decomposition
\[
G(z) = T_1(z) \begin{pmatrix} D(z) & 0 \\ 0 & 0 \end{pmatrix} T_2(z).
\]
Here $T_1,T_2$ are matrix polynomials with matrix polynomial inverses
 while 
 \[
 D(z)=\text{diag}(d_1(z),\dots,d_r(z))
 \]
  is an $r\times r$ diagonal matrix 
with only non-zero polynomials on the diagonal.
Notice that $T(z)$ has rank $r$ whenever $\det D(z)\ne 0$, $z\ne 0$.
Since $T$ is self-adjoint on $\T$, we have $T(z) = T(1/\bar{z})^*$ for $z\ne 0$
and so
\begin{equation} \label{alsohave}
T_2^{-1}(1/\bar{z})^* T(z) T_2^{-1}(z) = z^{-n} T_2^{-1}(1/\bar{z})^* T_1(z) \begin{pmatrix} D(z) & 0 \\ 0 & 0 \end{pmatrix} 
= z^n \begin{pmatrix} D(1/\bar{z})^* & 0 \\ 0 & 0 \end{pmatrix} T_1(1/\bar{z})^*T_2^{-1}(z)
\end{equation}
is a matrix Laurent polynomial which is positive semi-definite on $\T$ 
and with $0$ in the last $N-r$ columns and rows.  Thus, \eqref{alsohave}
has the form $\begin{pmatrix} T_0(z) & 0 \\ 0 & 0 \end{pmatrix}$
where $T_0$ is an $r\times r$ matrix Laurent polynomial which is positive semi-definite on $\T$
and crucially satisfying $\det T_0 \not\equiv 0$ since $T$ has rank $r$
outside of a finite set.

By Theorem \ref{FR}, there exists an $r\times r$ matrix polynomial
$A_0$ such that $\det A_0(z) \ne 0$ in $\D$ and $A_0(z)^*A_0(z) = T_0(z)$ on $\T$.
If we set $V=T_2$ and
\[
A = \begin{pmatrix} A_0 & 0_{r\times (N-r)} \end{pmatrix} V
\]
then $A(z)^*A(z) = T(z)$ on $\T$.
Note that $A(1/\bar{z})^* A(z) = T(z)$ holds in $\C\setminus \{0\}$
since both sides are analytic and agree on $\T$.

Our degree bound on $A$ follows from the fact that
\[
z^{n} T(z) V(z)^{-1} \begin{pmatrix} A_0(z)^{-1} \\ 0 \end{pmatrix}  = z^n A(1/\bar{z})^*
\]
is analytic at $0$.  A right rational inverse of $A$ is given by $V^{-1} \begin{pmatrix} A_0^{-1} \\ 0 \end{pmatrix}$.
\end{proof}

The matrix Fej\'er-Riesz factorization 
described is maximal in the sense of the following
theorem.  One can also describe all other factorizations.
There is
nothing essentially new about this result, but
it is probably difficult to attribute.  It
could be deduced from inner-outer factorizations.

\begin{theorem} \label{genFR-uniqueness}
Assuming the setup and notation 
of Theorem \ref{genFR}.
For any other factorization $T = C^*C$ on $\T$ with a matrix polynomial $C$, 
there exists a rational iso-inner 
function $\Phi$ such that $C = \Phi A$
(necessarily, $\Phi = CB$).
If $C$ has a right rational inverse holomorphic in $\D$ then $\Phi$ is a constant unitary matrix.
\end{theorem}

\begin{proof}
Suppose $T = C^* C$ on $\T$.  Then, we may write $C V^{-1} = \begin{pmatrix} C_0 & C_1\end{pmatrix}$ 
where $C_0$ has $r$
columns.  Since 
\[
(V^{-1})^* A^*A V^{-1} = \begin{pmatrix} A_0^*A_0 & 0 \\ 0 & 0 \end{pmatrix}
= \begin{pmatrix} C_0^*C_0 & C_0^*C_1 \\ 
C_1^*C_0 & C_1^*C_1 \end{pmatrix} = (V^{-1})^* C^*C V^{-1}
\]
we see that $C_0^*C_0 = A_0^*A_0$, $C_1^*C_1 = 0$ on $\T$.  This implies $C_1 \equiv 0$.
Then, $\Phi := C_0 A_0^{-1}$ is analytic on $\D$ and isometry-valued on $\T$.
Any poles on $\T$ are necessarily removable because 
$\Phi$ is rational and bounded on $\T$.  
We also have $\Phi A = C$.  
If $C$ has right rational inverse $C'$ then $\Phi AC' = I$. 
An isometry can only have a right inverse if it is
square, so $\Phi$ must be square (hence
unitary on $\T$) and $AC'$ must be unitary-valued on $\T$.
By the maximum principle, $\Phi$ and $AC'$ are contractive
in the disk; however, since they are inverses of each other they must be 
unitary-valued in the disk.  
Such analytic functions are constant. 
(Lemma \ref{Schurbreakup} 
proves something more general 
than this.)
\end{proof}

We now sketch a simple proof of Dritschel's positive definite
multivariable
Fej\'er-Riesz result (Thm \ref{dritschel1}).
Although it borrows elements from the original 
proof, we think it has some nice efficiencies in
exposition.     

\begin{proof}[Proof of Theorem \ref{dritschel1}]
Let $n$ be a positive integer and define the multivariable Cesaro
summation operator $C_n$ which we apply
to $N\times N$ 
matrix Laurent polynomials $L(z) = \sum_{k\in \Z^d} L_k z^k$
\[
(C_n L)(z) = \sum_{k\in \Z^d} c^n_k L_k z^k = \int_{\T^d} F_n(z,\zeta) L(\zeta) d\sigma(\zeta)
\]
where 
\[
c^n_k =\begin{cases} \prod_{j=1}^{d} \frac{n-|k_j|}{n} & \text{ for } |k_1|,\dots,|k_d|\leq n \\
0 & \text{ otherwise } \end{cases},
\]
\[
F_n(z,\zeta) = \frac{1}{n^d}\prod_{j=1}^d \left|\frac{1-z_j^n\bar{\zeta}_j^n}{1-z_j\bar{\zeta}_j}\right|^2 = \sum_{k\in \Z^d} c^n_k z^k
\]
is the Fej\'er kernel and $d\sigma$ is normalized Lebesgue measure on $\T^d$.

Let $\mcL_m$ be the vector space of $N\times N$ Laurent polynomials
of degree at most $m$ in each variable separately.
We shall consider $C^m_n := C_n|_{\mcL_m}: \mcL_m\to \mcL_m$.  
By basic properties of Cesaro summation, $C^m_n L \to L$ uniformly on
$\T^d$ as $n\to \infty$ for $L\in \mcL_m$.
Since the set of linear operators $B(\mcL_m)$ on $\mcL_m$ is
finite dimensional, $C^m_n$ tends to the identity as $n\to \infty$ with respect
to any norm on $B(\mcL_m)$. 
In particular, for $n$ large enough $C^m_n$ is invertible
and $(C^m_n)^{-1}$ tends to the identity as $n\to \infty$.

We next point out that if $L\in \mcL_m$ is positive semi-definite
on $\T^d$ then $C^m_n L$ is a sum of squares.
The reason is that on $\T^d$, $F_n(z,\zeta) L(\zeta)$
is a Laurent polynomial of degree at most $n+m$ with respect
to $\zeta$.  Then, the integral representation of $C_n L$ can be
computed via ``quadrature.''  
Indeed, for any $M$, if $H \in \mcL_M$ and $\mu = e^{2\pi i /(M+1)}$ then
\[
\int_{\T^d} H(\zeta) d\sigma(\zeta) = \frac{1}{(M+1)^d} \sum_{0\leq j_1,\dots, j_d \leq M} H(\mu^{j_1},\dots, \mu^{j_d}).
\]
This can be proven by testing on monomials.
This means that $C_n L(z)$ is a positive
finite linear combination of the terms $F_{n}(z, (\mu^{j_1},\dots, \mu^{j_d}))L(\mu^{j_1},\dots, \mu^{j_d})$.
Since $F_n$ is evidently a squared polynomial and each value of $L$ on $\T^d$ is assumed positive
semi-definite, we see that $C_n L$ is a sum of squares of polynomials.

Now, let $T \in \mcL^m$ be strictly positive on $\T^d$, 
i.e. there exists $\delta>0$ such that
$T(z)\geq \delta I$ for $z\in \T^d$.
For $n$ large enough, $T_n := (C^m_n)^{-1} T$ is also strictly positive.
Then, $T = C_n T_n$ is a Cesaro sum of a positive Laurent polynomial
which was already shown to be a sum of squares.
\end{proof}

\subsection{PSD kernels} \label{PSDsec}

We now discuss the proof of Lemma \ref{Skernel-lemma}
which claims that for $S:\D\to\C^{M\times N}$ analytic
and $\|S(z)\|\leq 1$ in $\D$ we have
that 
\[
K_S(w,z) = \frac{I_N-S(w)^*S(z)}{1-\bar{w}z}
\]
is positive semi-definite (PSD).
Let us recall the abstract definition of PSD for matrix or operator-valued
kernels.

\begin{definition} \label{def:PSD}
Let $X$ be a set, $\mathcal{L}$ a complex Hilbert space,  
and $K:X\times X \to B(\mathcal{L})$
a function; here $B(\mathcal{L})$ is
the set of bounded linear self-maps of $\mathcal{L}$.  
We say that $K$ is a \emph{PSD kernel} if
 for any $x_1,\dots, x_n \in X$ and $v_1,\dots, v_n \in \mathcal{L}$
we have
\[
\sum_{i,j} \langle K(x_i,x_j) v_j,v_i\rangle \geq 0.
\]
\end{definition}

Notice that if $(x,y) \mapsto K(x,y)$ is a PSD kernel, then
$(x,y) \mapsto K(y,x)$ is not necessarily PSD except
in the scalar case $\mathcal{H} = \C$.

\begin{definition}  \label{rank}
The \emph{rank} of $K$ is the maximum of the 
ranks of the block operators $(K(x_i,x_j))_{i,j}$
as we vary over $n$ and $x_1,\dots, x_n \in X$.
\end{definition}

\begin{proof}[Proof of Lemma \ref{Skernel-lemma}]
Our proof uses rudiments of
vector-valued Hardy spaces on the unit disk.
See Agler-McCarthy \cite{AMbook} for details.

Let $H_{M} = H^2(\D)\otimes\C^M$ be the set of $M$-dimensional column vectors with entries
in the Hardy space on the unit disk $H^2(\D)$.  
Left multiplication by $S$, $M_S:H_N\to H_M$, 
is contractive.  If $k_w(z) = k(z,w) := \frac{1}{1-\bar{w}z}$ is 
the Szeg\H{o} kernel, then by a fundamental
formula in reproducing kernel Hilbert space theory
\[
M_S^* (k_w \otimes v) = k_w \otimes S(w)^*v
\]
for $v\in \C^M$.
We see that
\[
\langle (I-M_S M_S^*) (k_w\otimes v_1), k_z \otimes v_2 \rangle_{H_M} =
k(z,w) \langle (I-S(z)S(w)^*)v_1, v_2\rangle_{\C^M}
\]
which after a short calculation using the fact that $I-M_S M_S^* \geq 0$ shows 
$(z,w) \mapsto \frac{I-S(z)S(w)^*}{1-z\bar{w}} \text{ is PSD}.$
We could apply the same argument to $\breve{S}(z):= S(\bar{z})^*$ 
to see that 
$(z,w) \mapsto \frac{I - S(\bar{z})^* S(\bar{w})}{1-z\bar{w}} \text{ is PSD.}$
Replace $z,w$ with their conjugates and relabel the
variables to see that $K_S(w,z)$ is PSD.
\end{proof}

\begin{proof}[Proof of Proposition \ref{degdet}]
Assuming $S:\D\to \C^{N\times N}$ is rational inner we need to compute the rank of the positive
semi-definite kernel $(w,z)\mapsto \frac{I- S(w)^*S(z)}{1-\bar{w} z}.$
We shall use notation from the proof of Lemma \ref{Skernel-lemma} above.
As in said proof,
it is notationally easier to deal with the
kernel 
\[
K(z,w) = \frac{I-S(z) S(w)^*}{1-z\bar{w}}
\]
and we can reduce to this case by 
replacing $S$ with $S(\bar{z})^*$.

Now, $K$ is the reproducing kernel for $H_N\ominus S H_N$.
This follows from the fact that $S$ is inner: $SH_N$
is a closed subspace of $H_N$ and has reproducing kernel
\[
\frac{S(z)S(w)^*}{1-z\bar{w}}
\]
which can be verified by the following calculation
\[
\langle Sf, k_w S S(w)^*v \rangle_{H_N} = \langle f, k_w S(w)^*v\rangle_{H_N} = \langle S(w)f(w), v \rangle_{\C^N}
\]
for $f \in H_N$.
The rank of $K$ is the dimension of $H_N\ominus SH_N$.

To count this dimension we write $S =Q/p$ in lowest terms.
Since $S$ is bounded on $\T$ it can have no poles on $\T$,
and therefore $p$ has no zeros in $\overline{\D}$.
Let $Q(z) = T_1(z) D(z) T_2(z)$ be the Smith
normal form decomposition for $Q$ (Theorem \ref{SNF} above).  
Notice
that $D$ has full rank
on $\T$ since $S$ is inner.
Write $D = \text{diag}(d_1,\dots, d_N)$.
Then, $\det Q = c \det D = c \prod_j d_j$
where $c = \det T_1 \det T_2$ is a constant
because $T_1,T_2$ have
polynomial inverses.  
Since $S$ is inner $\det S = \frac{\det Q}{p^N}$
is a finite Blaschke product.
Its degree equals its number of zeros in $\D$
which equals the number of zeros of $\det Q$ in
$\D$ since $p$ has none.

The vector space $H_N \ominus S H_N$
is isomorphic to the vector space quotient
\[
H_N / S H_N = H_N / (T_1 D T_2) H_N = H_N/(T_1 D H_N)
\cong H_N/ D H_N.
\]
The first equality holds because $p$ has no zeros in $\overline{\D}$,
the second holds because $T_2$ has a polynomial inverse,
and the last isomorphism holds because $T_1$ has a polynomial
inverse.
Recalling $D = \text{diag}(d_1,\dots, d_N)$
we note the dimension of $H^2/ d_j H^2$ 
is the number of zeros of $d_j$ in $\D$
and therefore
the dimension of $H_N/ D H_N$
is the number of zeros of $\prod_{j=1}^{N} d_j$
inside $\D$ (counting multiplicities).
\end{proof}

This proof appears in \cite{BickelKnese}.

\end{document}